\documentclass[11pt,a4paper,twoside]{article}
\usepackage{amssymb,amsmath,amsfonts,amsthm,hyperref,graphicx,color}
\usepackage{fullpage}
\usepackage[T1]{fontenc}
\usepackage[utf8]{luainputenc}

\newcommand{\region}[1]{\ensuremath{\mathbf{#1}}}
\newtheorem{theorem}{Theorem}
\newtheorem{lemma}{Lemma}

\begin{document}
\title{The limit of $L_p$\,Voronoi diagrams as $p\rightarrow0$ is the bounding-box-area Voronoi diagram}
\author{Herman Haverkort and Rolf Klein\\Department of Computer Science, University of Bonn}
\date{14 July 2022}
\maketitle

\begin{abstract}
\noindent
We consider the Voronoi diagram of points in the real plane when the distance between two points $a$ and $b$ is given by $L_p(a-b)$ where
$L_p((x,y))  =  (|x|^p+|y|^p)^{1/p}.$ 
We prove that the Voronoi diagram has a limit as $p$ converges to zero from above or from below: it is the diagram that corresponds to the distance function $L_*((x,y)) = |xy|$.
In this diagram, the bisector of two points in general position consists of a line and two branches of a hyperbola that split the plane into three faces per point. We propose to name $L_*$ as defined above the \emph{geometric $L_0$ distance}.

\end{abstract}

%%%%%%%%%%%%%%%%%%%%%%%%%%%%%%%%%%%%%%%%%%%%%%%%%%%%%%
\section{Introduction}\label{intro-sect}
%%%%%%%%%%%%%%%%%%%%%%%%%%%%%%%%%%%%%%%%%%%%%%%%%%%%%%
The $L_p$ distance between two points $a$ and $b$ is given by $L_p(a-b)$ where\begin{equation}
L_p((x,y))  = (|x|^p+|y|^p)^{1/p}.\label{Lpdef}
\end{equation}
The $L_p$ distances are widely known and used in computational geometry and applications; see Okabe, Boots, Sugihara and Chiu \cite{obsc-stcav-00}
or Klein, Driemel and Haverkort \cite{kdh-ag-22}. 
In this paper we analyse what happens as $p$ tends to zero. We see immediately that the values of the $L_p$ distance function in \ref{Lpdef} differ depending on whether $p$ tends to zero from above or from below: 
If $x \neq 0$ and $y \neq 0$, then $\lim_{p \downarrow 0} L_p((x,y)) = \infty$ whereas $\lim_{p \uparrow 0} L_p((x,y)) = 0$.
Nevertheless, in this paper, we find that, at least for points in general position, the limit of the Voronoi diagram under the $L_p$ distance as $p$ approaches zero is well-defined. In other words, even though, for small values of $|p|$, the $L_p$ and $L_{-p}$ distances have very different values, they induce almost identical Voronoi diagrams. 

To be precise, observe that the terms $|x|^p$ and $|y|^p$ In Equation~\ref{Lpdef} are undefined if $p < 0$ and $x=0$ or $y=0$, respectively. 
We extend the function $L_p$, for $p<0$, %as well as the function $L_*$, 
to a continuous function on all arguments $(x,y)$ by setting its value to zero if $x$ or $y$ is equal to zero. Thus, given two points $a$ and $b$ in the plane, we can calculate and compare their $L_p$ distances to any other point in the plane, for any $p \neq 0$.

Figure \ref{vorosketch-fig} shows an example of the Voronoi diagram of two points under the $L_p$ distance for two values of $p$ very close to zero.
The diagrams have been created with Vorosketch developed by the first author \cite{h-vs-22}.  The figure shows, in particular, the bisector $B_p(a,b)$ of two points $a = (a_x,a_y)$ and $b = (b_x,b_y)$, that is, it shows the set of points in the plane that are equidistant to $a$ and $b$ under the $L_p$ distance.
The bisector $B_p(a,b)$ separates the set $V_p(a,b)$ of points that are closer to $a$ from the set $V_p(b,a)$ of points that are closer to $b$.
In this paper we prove the following result.
\begin{figure} 
\begin{center}
\includegraphics[width=106.67pt]{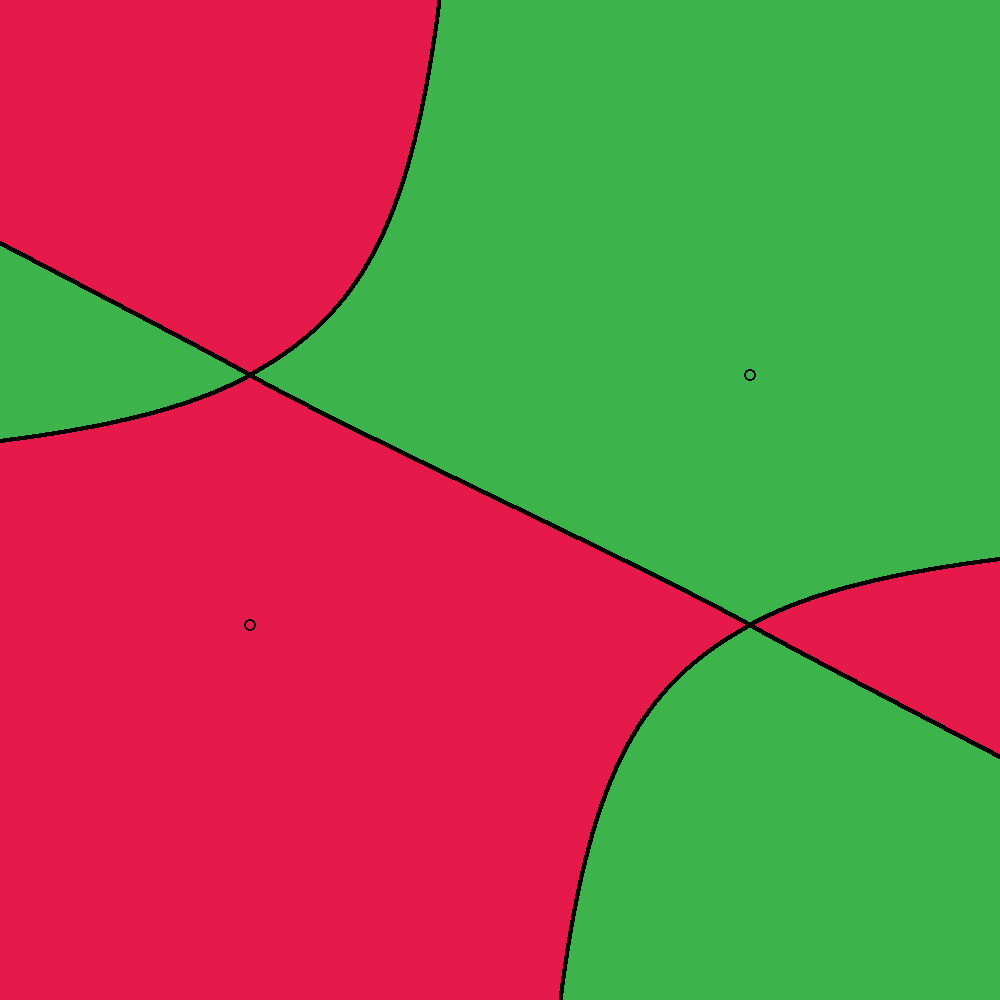}\quad\quad
\includegraphics[width=106.67pt]{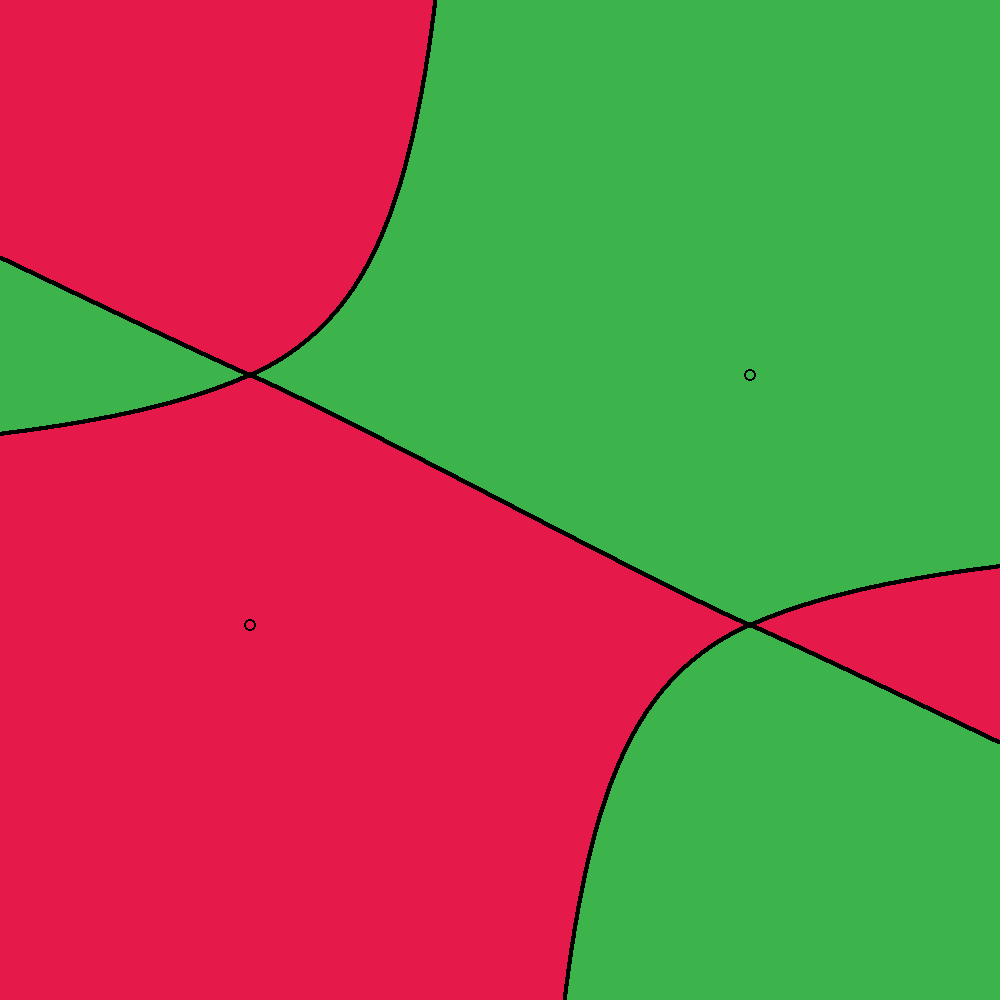}
\caption{A bisector (in black) of two points under the $L_p$ distance for $p = -0.05$ (left) and $p = 0.05$ (right), as computed pixelwise by Vorosketch \cite{h-vs-22}. The bisector divides the plane into two regions (points closer to $a$ and points closer to $b$) that each seem to consist of three faces.}
\label{vorosketch-fig}
\end{center}
\end{figure}
\begin{theorem}\label{main-theo}
Let $a$ and $b$ be two point sites in the real plane with different $X$- and $Y$-coordinates. Then 
their bisector $B_p(a,b)$ and the sets $V_p(a,b)$ and $V_p(b,a)$ % := \{ q \mid L_p(q,a) = L_p(q,b) \}$ 
under the $L_p$ distance converge, as $p$ tends to zero from above or from below, to their bisector $B_*(a,b)$ and the sets $V_*(a,b)$ and $V_*(b,a)$ under the $L_*$ distance as defined by:\begin{equation}
  L_*((x,y)) = |xy|.
  \label{L0def}
\end{equation}
\end{theorem}
Note that $L_*(a-b)$ is simply the area of the axis-parallel bounding box of $a$ and $b$. Let $\lambda$ and $\rho$ denote the other two vertices of their axis-parallel bounding box. 
\begin{figure}
\begin{center}
\includegraphics[width=\hsize]{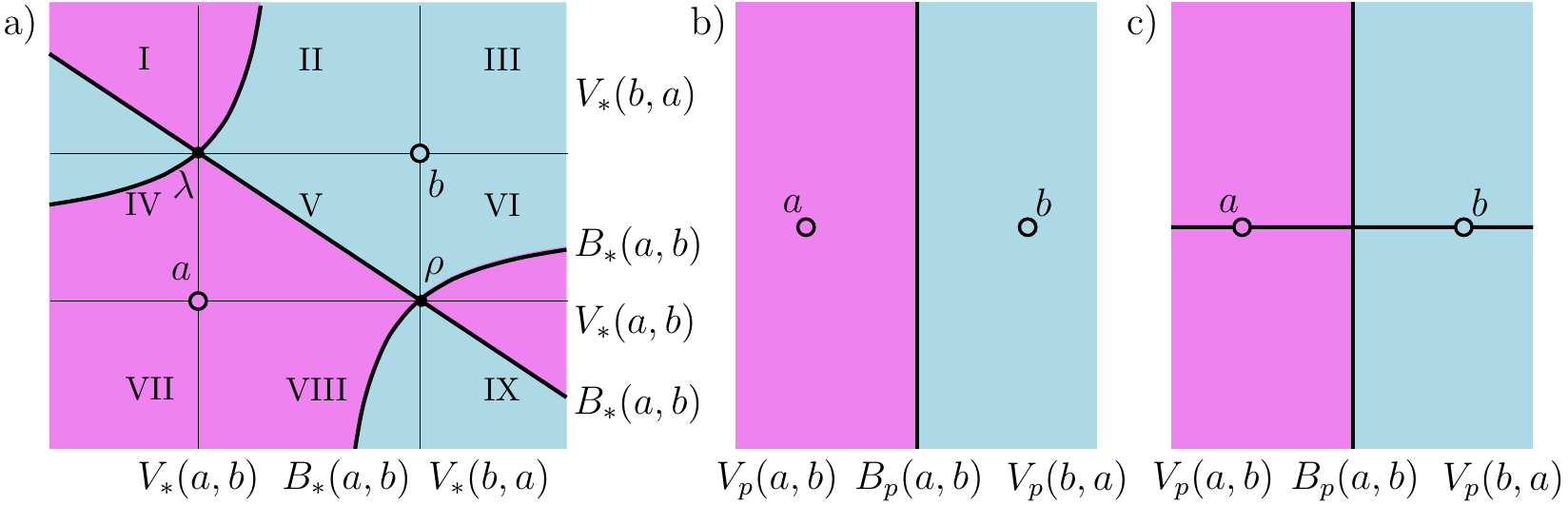}
\caption{(a) A bisector of two points $a$ and $b$ in $L_*$. (b) A bisector of two points on a horizontal line in $L_p$, where $p > 0$. (c) A bisector of two points on a horizontal line in $L_p$, where $p < 0$: note that the line through $a$ and $b$ does not lie on the common boundary of the regions $V_p(a,b)$ and $V_p(b,a)$, but it is part of the bisector $B_p(a,b)$ nonetheless, as all points on the line through $a$ and $b$ have distance 0 to both $a$ and $b$.}
\label{Lstarbisec-fig}
\end{center}
\end{figure}
The bisector $B_*(a,b)$ consists of the line through $\lambda$ and $\rho$
and of two hyperbola branches through $\lambda$ and $\rho$, respectively, whose asymptotes are the vertical and horizontal lines through the bounding box's centre; see Figure~\ref{Lstarbisec-fig}(a). This can easily be verified by solving the equation $L_*(a-q) = L_*(b-q)$ for $q$ in each of the nine regions that result from subdividing the plane by the axis-parallel lines through $a$ and $b$: in the odd-numbered regions, the solution is the line with equation:\begin{equation}
(a_y-b_y)x + (a_x-b_x)y = a_x a_y - b_x b_y,\label{linedef}
\end{equation}
whereas in the even-numbered regions, the solution is the hyperbola with equation:\begin{equation}
\left(x - \frac12(a_x+b_x)\right)\left(y - \frac12(a_y+b_y)\right) = -\frac14(a_x-b_x)(a_y-b_y).\label{hyperboladef}
\end{equation}
The bisector $B_*(a,b)$ divides the plane into six faces, such that each face is entirely contained in either $V_*(a,b)$ or $V_*(b,a)$, the face containing $a$ lies in $V_*(a,b)$, the face containing $b$ lies in $V_*(b,a)$, and each point of $B_*(a,b)$ lies on the boundary of both $V_*(a,b)$ and $V_*(b,a)$.

It follows that for point sites in general position (that is, if no two points are on a common horizontal or vertical line), the limit of their $L_p$\,Voronoi diagram as $p$ tends to zero is well-defined and it equals the $L_*$\,Voronoi diagram. 
Thus, it appears that defining $L_0((x,y))$ as $L_*((x,y)) = |xy|$, which equals $\exp(\ln |x| + \ln |y|)$ if $x, y \neq 0$, constitutes a natural interpretation of $L_p((x,y)) = (|x|^p + |y|^p)^{1/p}$ for $p = 0$. Therefore we propose to call the $L_*$ distance measure the \emph{geometric $L_0$ distance},
and the resulting Voronoi diagram the {\em geometric $L_0$\,Voronoi diagram}, %
distinguishing 
it from other, and unrelated, definitions of $L_0$ \cite{wiki22}. 
The natural generalisation to higher dimensions would be to define $L_0(x)$, for a $d$-dimensional distance vector $x$, by $L_0(x) := |\prod_{i=1}^d x_i|$. 
Figure~\ref{8points-fig} shows an example of an $L_0$\,Voronoi diagram for eight points in the plane.
For $n$ points in general position, the complexity of this Voronoi diagram is in $\Omega(n^2)$,
as Seidel has observed\footnote{\label{s-cloVD-22}Personal communication by Raimund Seidel}; see Section \ref{circ-sect}. Figure~\ref{figallp} shows, for comparison, the $L_p$\,Voronoi diagram of the same sites for several values of~$p$.

\begin{figure}
\begin{center}
\includegraphics[width=240pt]{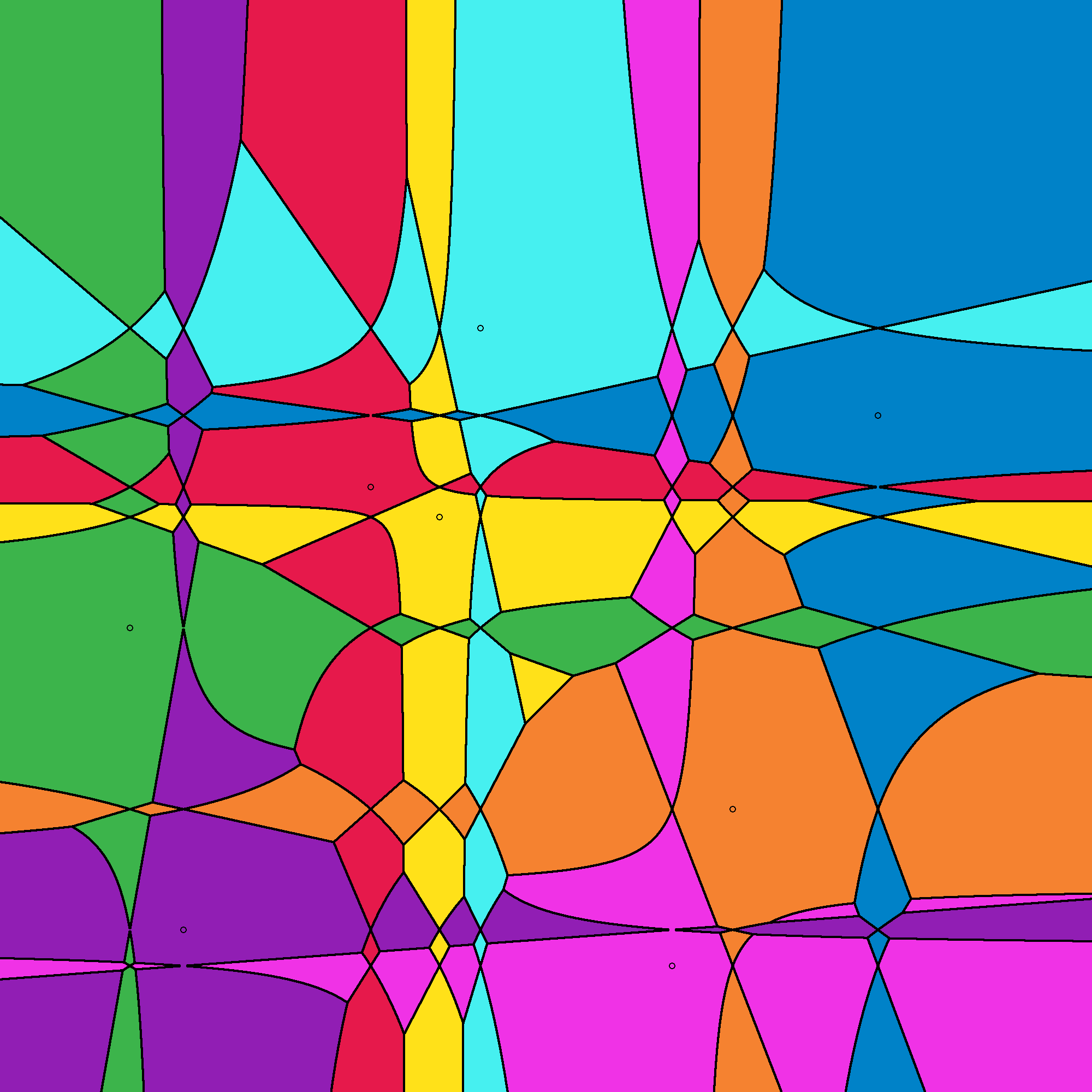}
\caption{An $L_0$\,Voronoi diagram as computed pixelwise by Vorosketch \cite{h-vs-22}.}
\label{8points-fig}
\end{center}
\end{figure}

\begin{figure}
\newcommand\VDp[2]{%
\vbox{\hsize=0.32\hsize\hbox{\includegraphics[width=\hsize]{8pointsL#2}}\hbox{$p=#1$}}%
}
\begin{center}
\leavevmode
\vbox{\hsize=0.85\hsize
\hbox to\hsize{%
\VDp{-\infty}{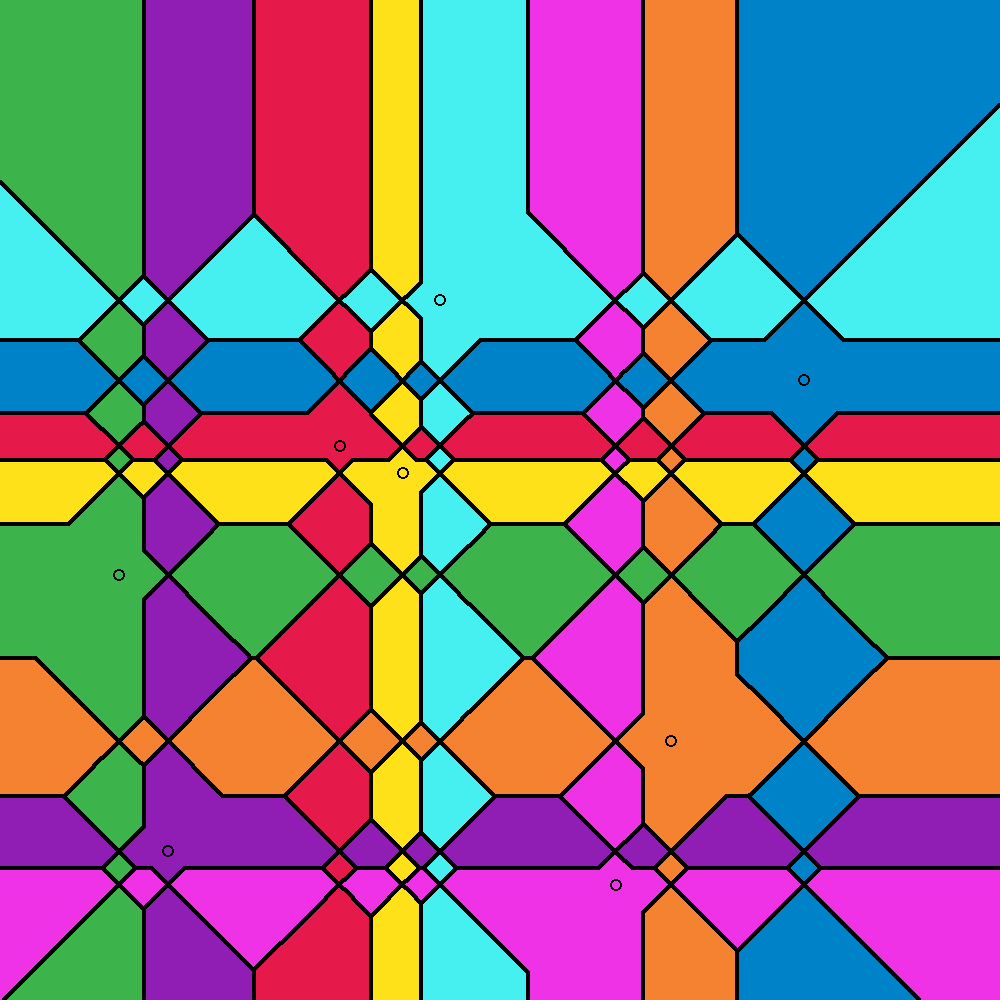}\hfill
\VDp{-1}{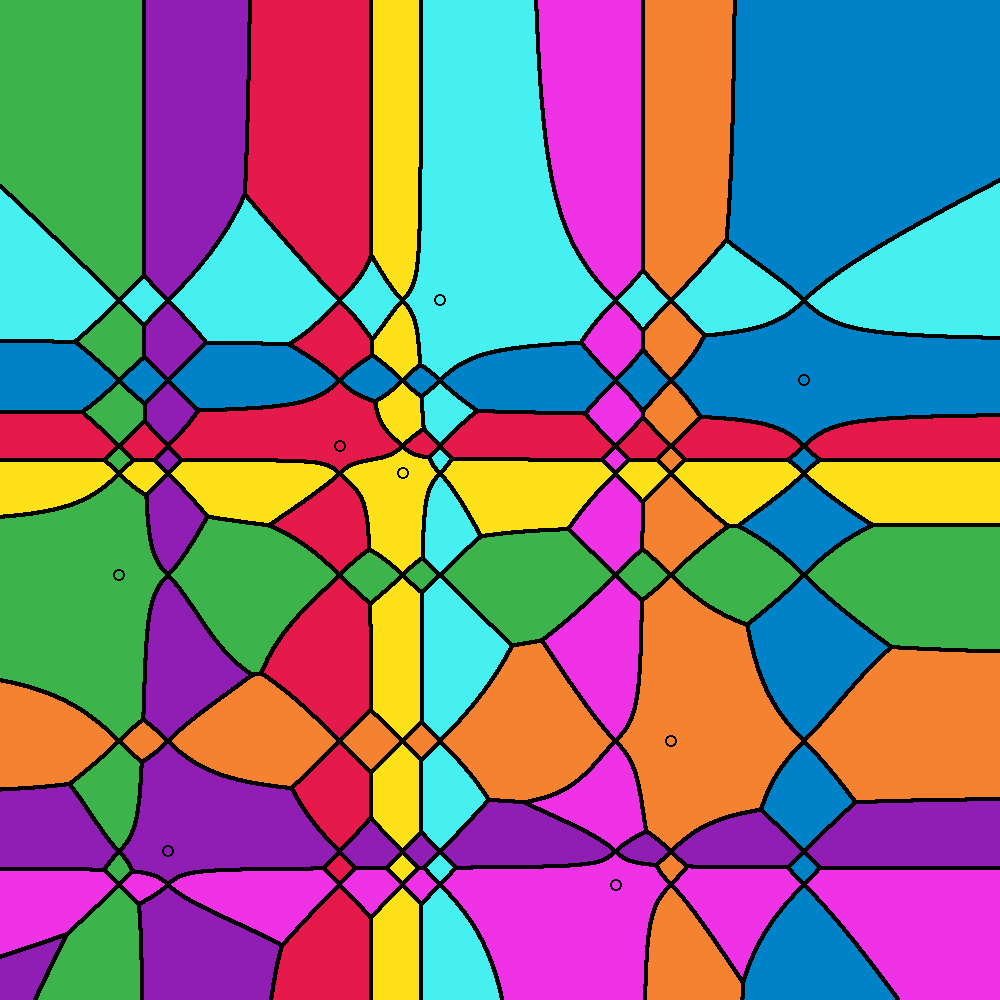}\hfill
\VDp{-0.2}{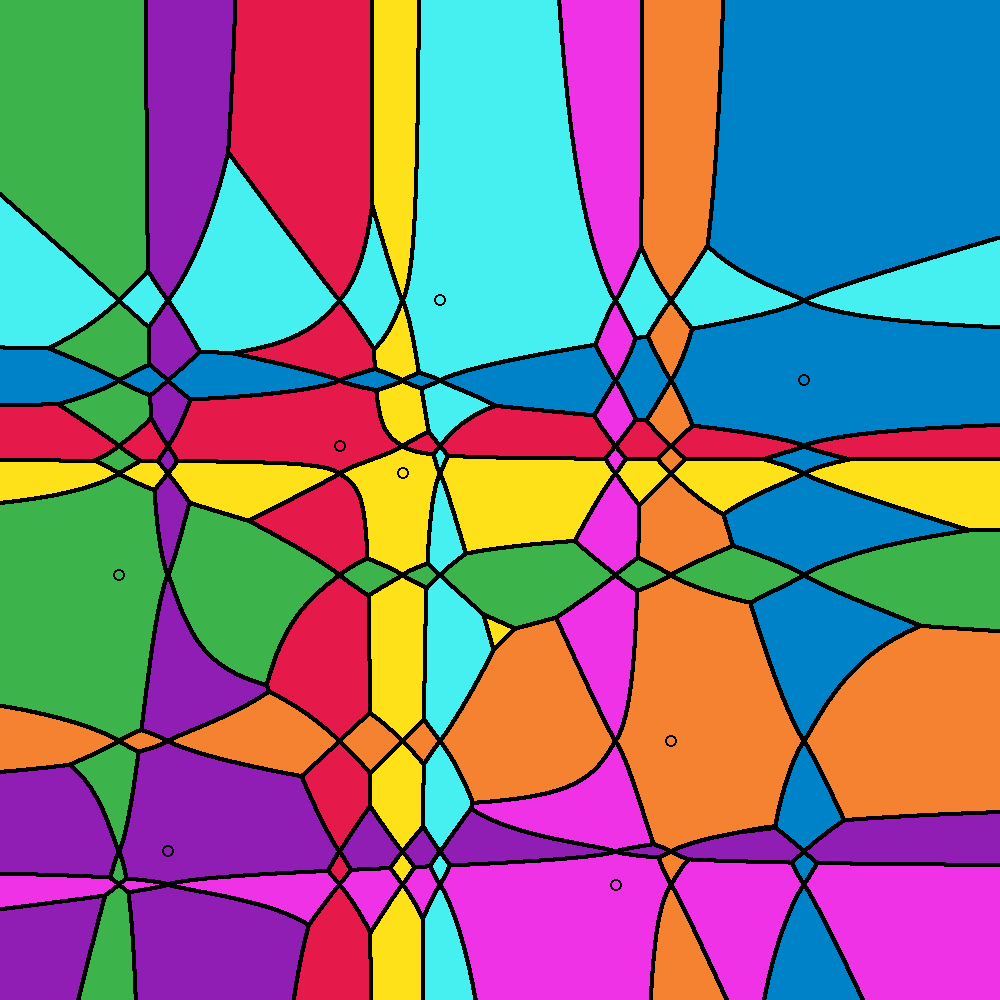}}
\kern0.04\hsize
\hbox to\hsize{%
\VDp{-0.1}{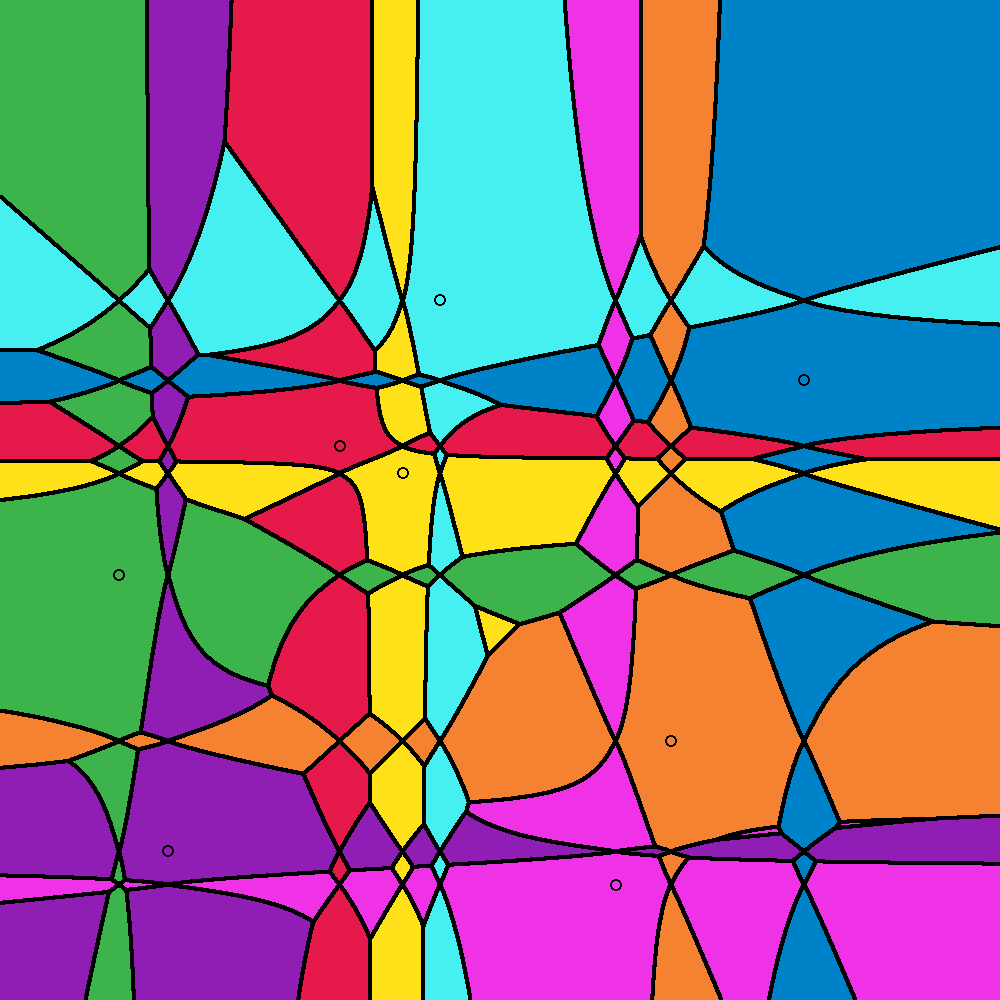}\hfill
\VDp{0}{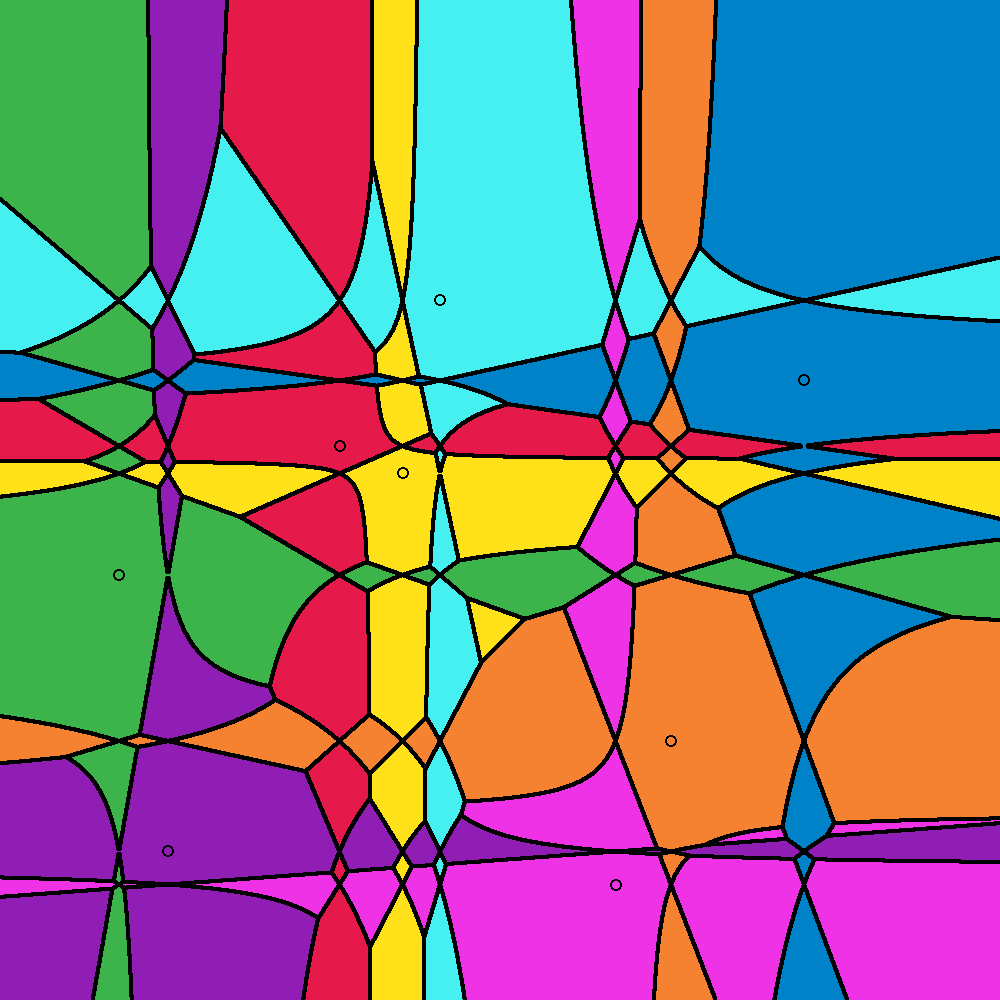}\hfill
\VDp{0.1}{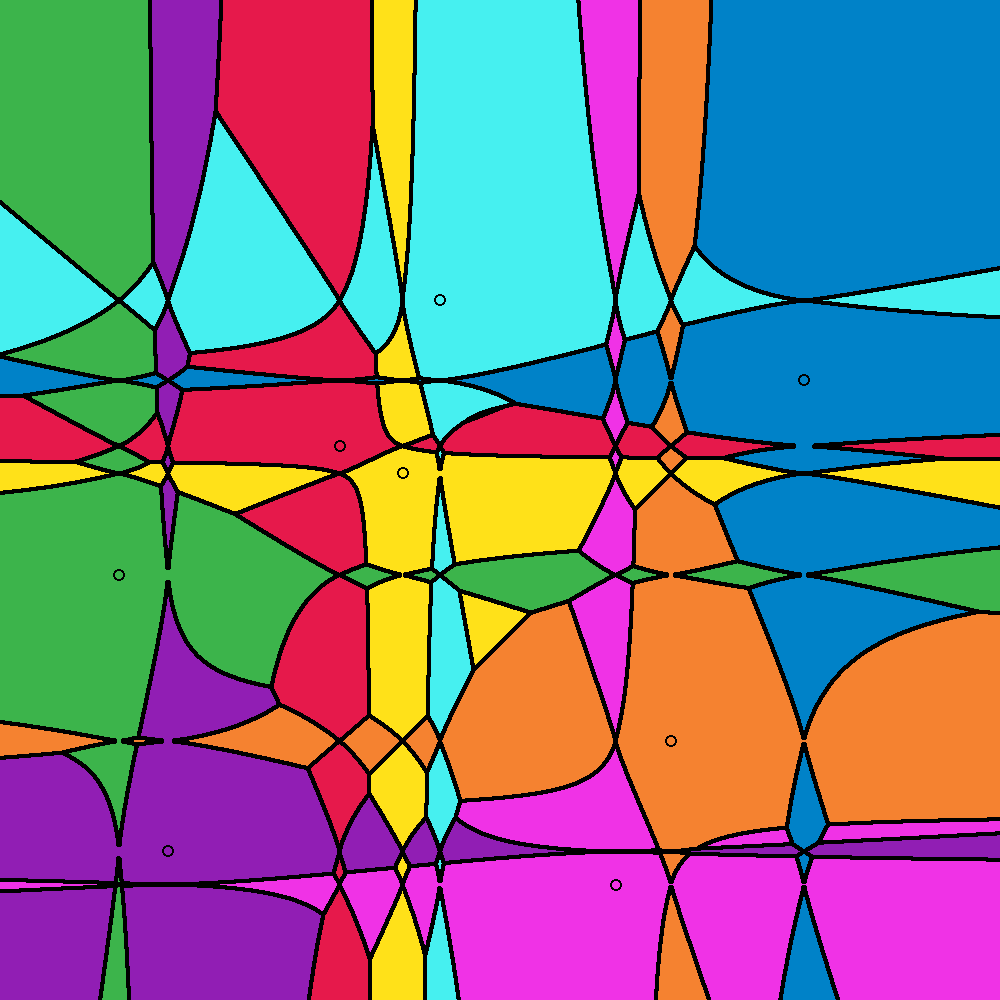}}
\kern0.04\hsize
\hbox to\hsize{%
\VDp{0.2}{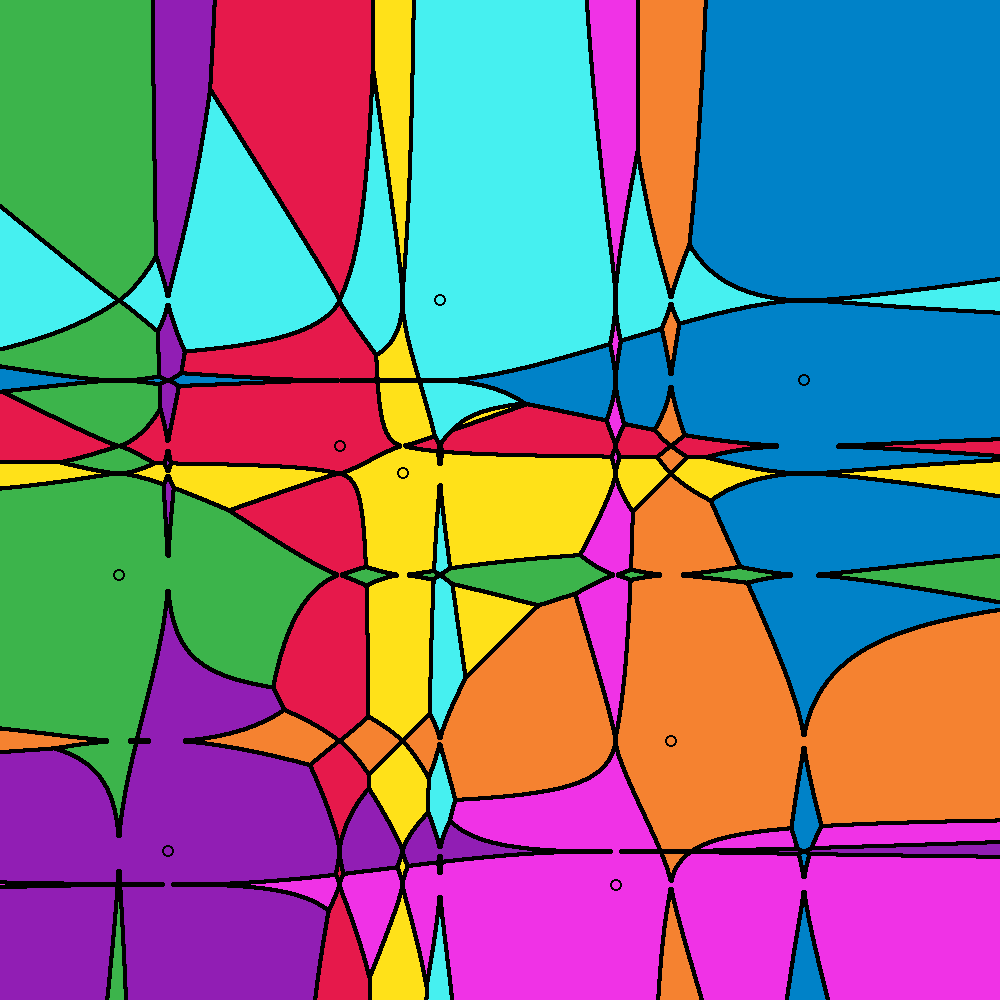}\hfill
\VDp{0.5}{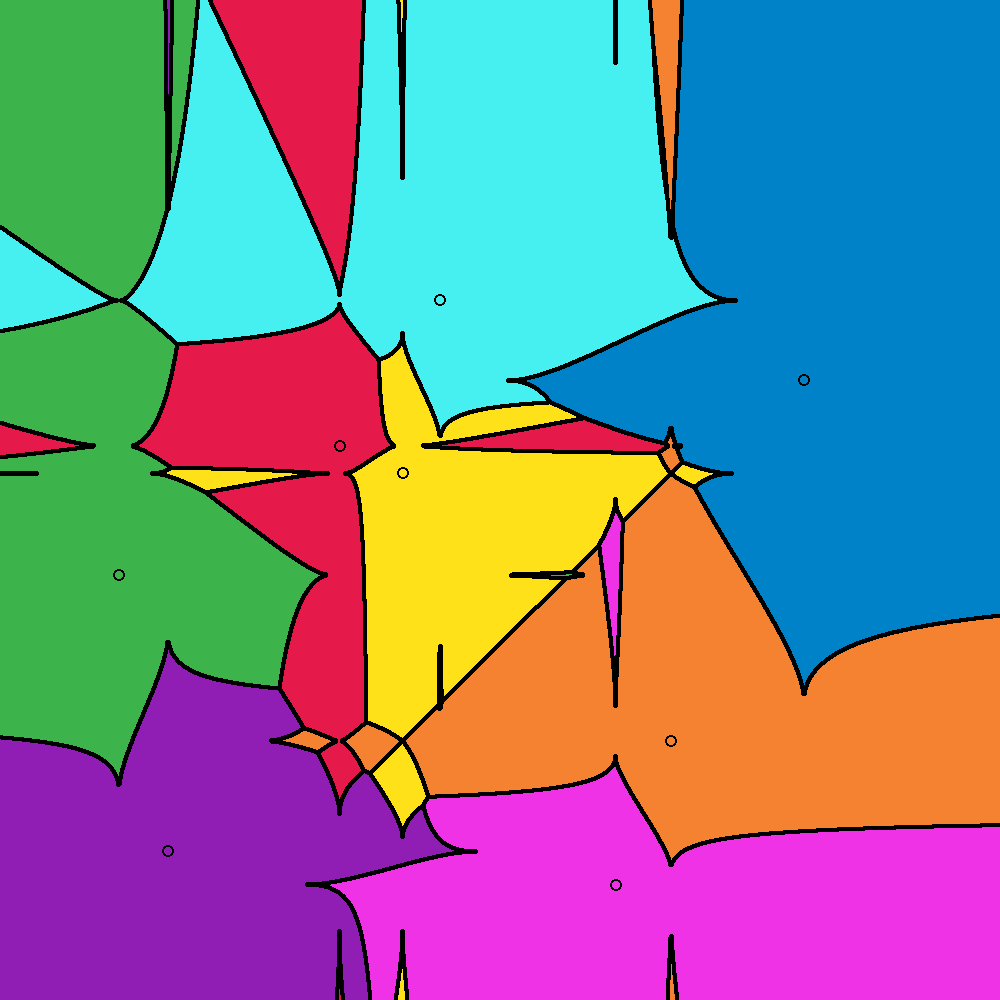}\hfill
\VDp{0.7}{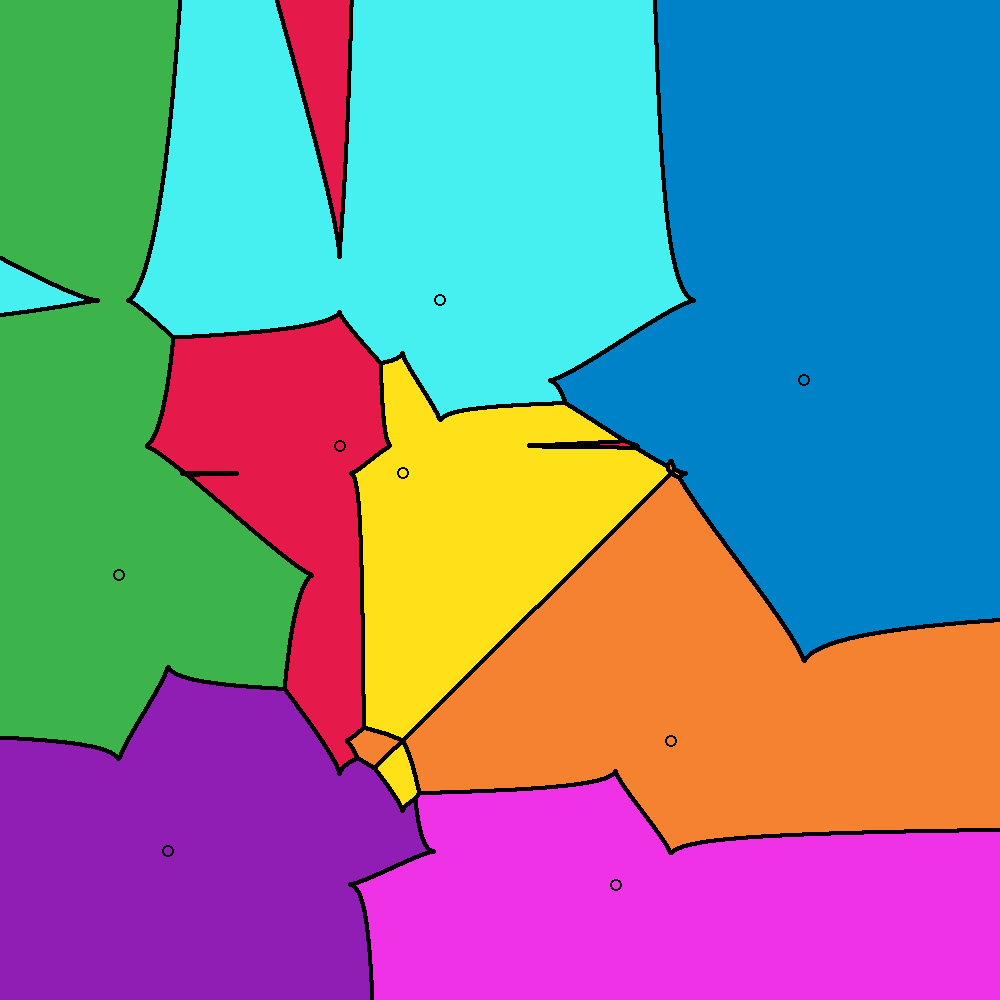}}
\kern0.04\hsize
\hbox to\hsize{%
\VDp{1}{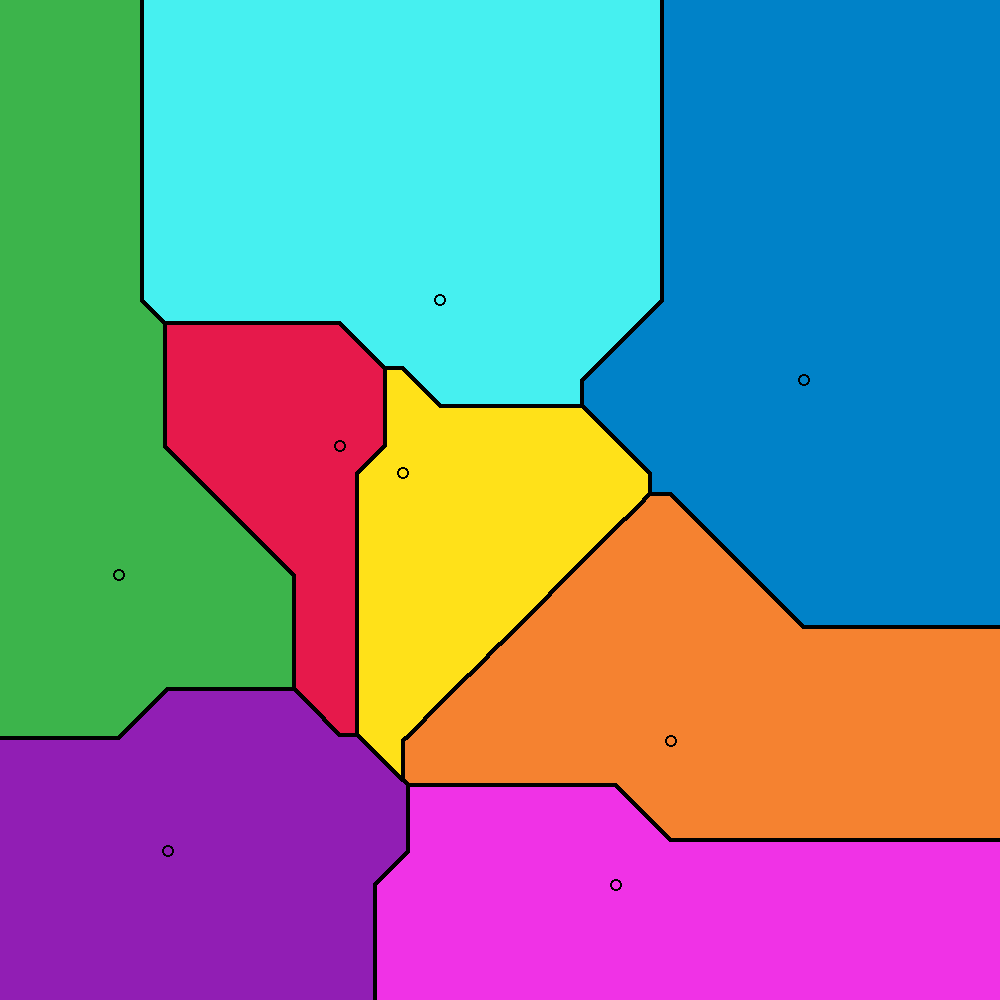}\hfill
\VDp{2}{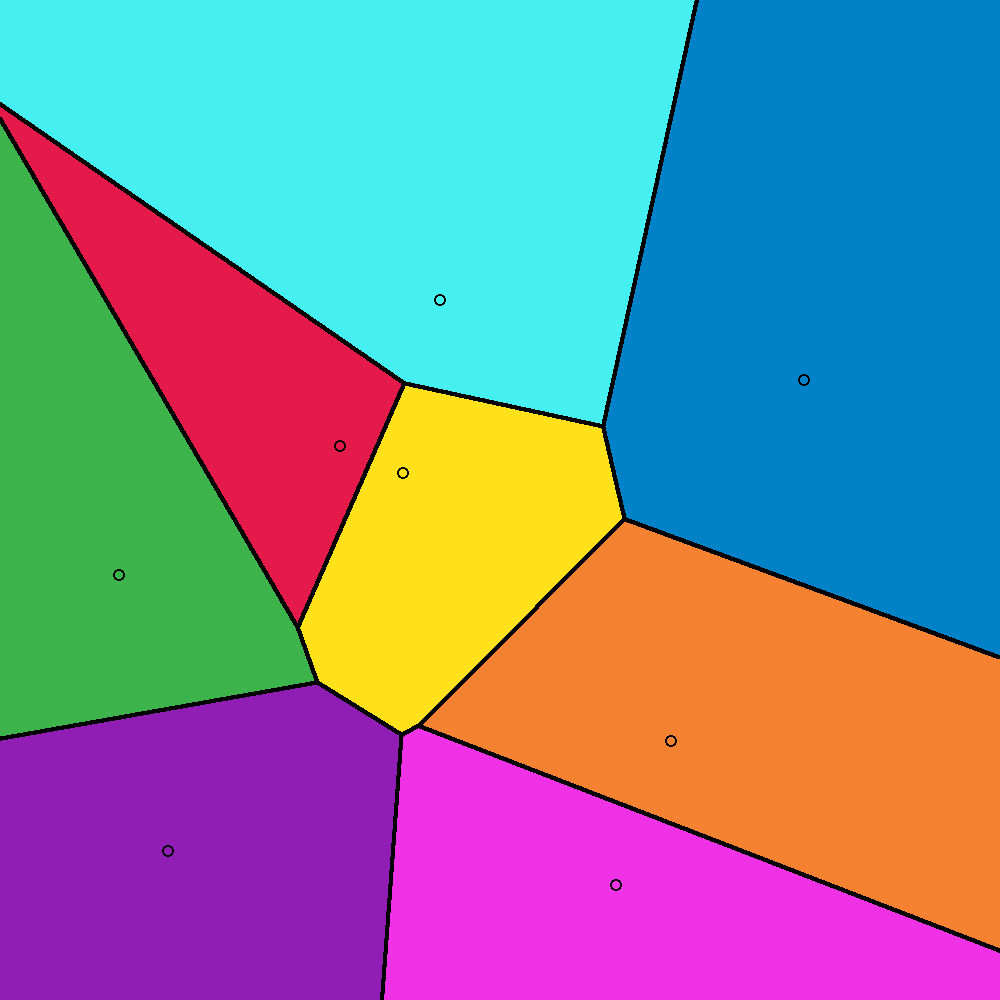}\hfill
\VDp{\infty}{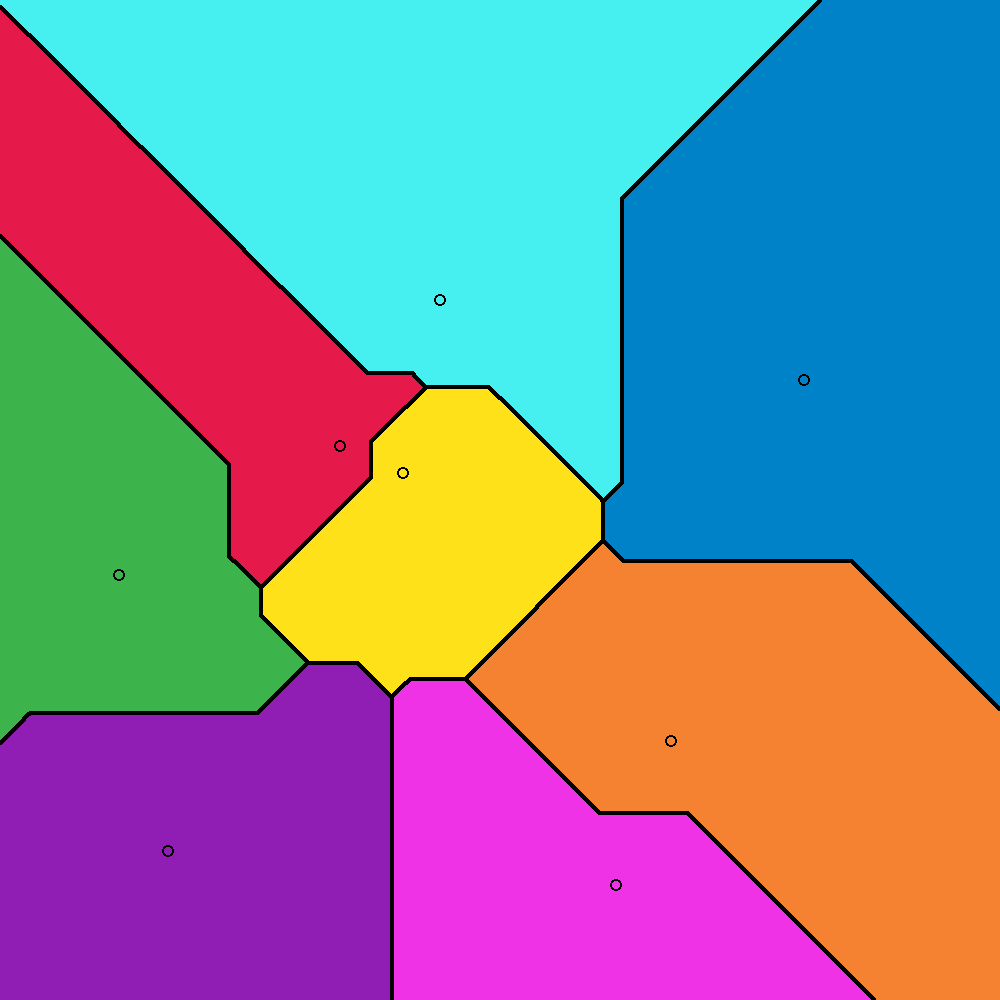}}}
\caption{Voronoi Diagrams under the $L_p$ distance for various values of $p$, as rendered pixel by pixel by Vorosketch~\cite{h-vs-22}. Note that features narrower than a pixel might not have been detected completely, so the combinatorial structure of the bisectors should be interpreted with care.}
\label{figallp}
\end{center}
\end{figure}

Note that $L_0(q) = \frac12 (L_1^2(q) - L_2^2(q))$, so the geometric $L_0$ distance measures, in a way, the difference between the $L_1$ and the $L_2$ distance.
Also, $\ln(L_0((x,y)))=
\ln |x| + \ln |y|$ could, perhaps, be considered the proper distance measure in Manhattan if cabs were accelerating exponentially fast.

If $a$ and $b$ lie on a common vertical or horizontal line, then $B_p(a,b)$ is discontinuous at $p = 0$,
as $\lim_{p\downarrow 0} B_p(a,b)$ is the line perpendicular to and through the midpoint of the segment $ab$, whereas
$\lim_{p\uparrow 0} B_p(a,b) = B_*(a,b) $ consists of both axis-parallel lines through the midpoint of $ab$ (see Figure~\ref{Lstarbisec-fig}(c)).

The proof of Theorem \ref{main-theo} is based on Lemma \ref{cell-lem} in Section~\ref{circ-sect}, 
which tells us in which parts of the plane $L_p$ bisectors appear, and on Lemma \ref{conv-theo} in Section~\ref{ana-sect}, 
which states that these $L_p$ bisector points converge to a line and a hyperbola.

%%%%%%%%%%%%%%%%%%%%%%%%%%%%%%%%%%%%%%%%%%%%%%%%%%%%%%
\section{$L_p$ circles}\label{circ-sect}
%%%%%%%%%%%%%%%%%%%%%%%%%%%%%%%%%%%%%%%%%%%%%%%%%%%%%%
First we want to show how a bisector structure as shown in Figure \ref{vorosketch-fig}
can come about. Metrical bisectors are written out by intersection points of expanding 
$L_p$ circles, so we take a look at those.

For each $p > 0$ and each radius $r > 0$, the $L_p$ circle of radius $r$ centred at the origin has \emph{corner points} $(0,\pm r)$ and $(\pm r, 0)$, since $L_p(0, \pm r) = (|0|^p + |r|^p)^{1/p} = r$ and $L_p(\pm r, 0) = (|r|^p + |0|^p)^{1/p} = r$.

As $p$ decreases from $2$ to $1$, the round Euclidean $L_2$ circle gets flattened into a diagonal
square that is just convex. As $p$ becomes smaller than $1$, the circle's segments between the
cornerpoints bend inwards; since convexity is lost, 
the triangle inequality no longer holds.
\begin{figure} 
   \begin{minipage}{.45\linewidth} % [b] => Ausrichtung an \caption
      \includegraphics[width=\linewidth]{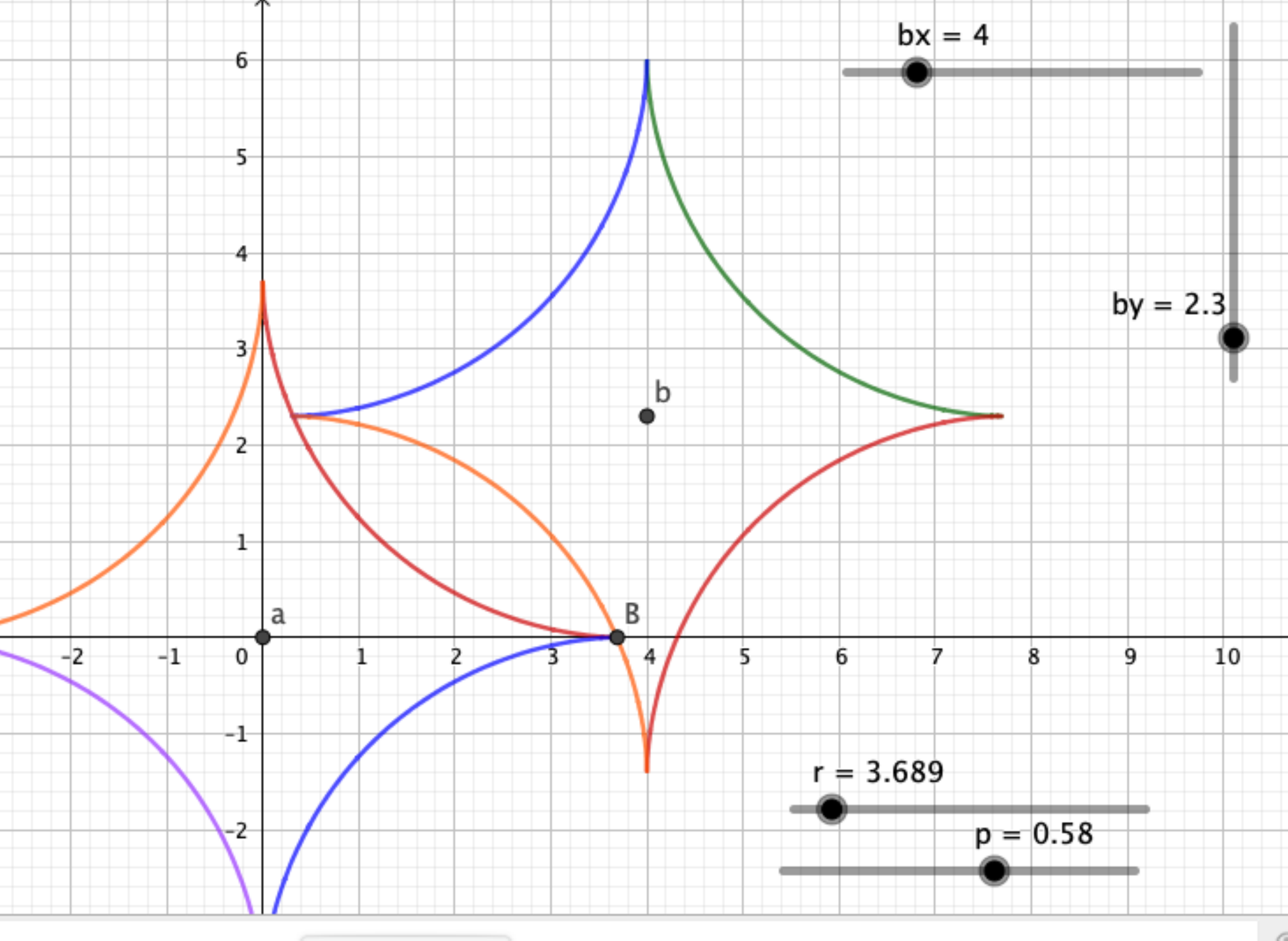}
      \caption{Two $L_{0.58}$ circles emanating from centres $a$ and $b$.}
      \label{lowp-fig}
   \end{minipage}
   \hspace{.1\linewidth}% Abstand zwischen Bilder
   \begin{minipage}{.45\linewidth}% [b] => Ausrichtung an \caption
      \includegraphics[width=\linewidth]{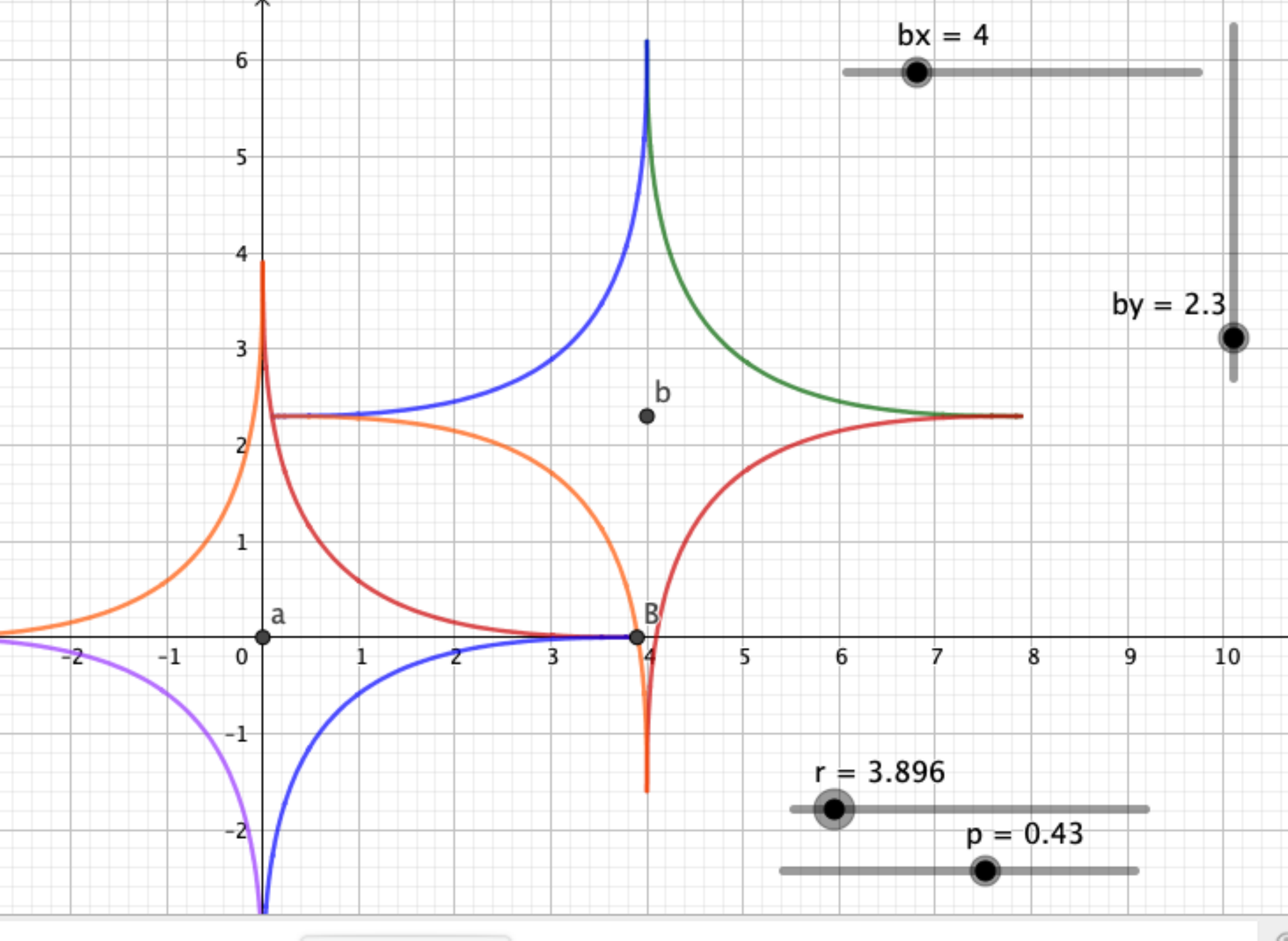}
      \caption{Bisector point $B$ moves to the right as $p>0$ shrinks.}
      \label{lowerp-fig}
   \end{minipage}
\end{figure}

Figure \ref{lowp-fig} shows two $L_{0.58}$ circles emanating from centres $a$ and $b$. At radius $r=3.689$ they meet
for the first time, creating two bisector points on the top and bottom edges of the bounding box of $a,b$.
Let us focus on the bisector point $B$
on the bottom edge of the bounding box.
As $p$ shrinks to $0.43$, point $B$ moves to the right along the bottom edge of the bounding box, as Figure \ref{lowerp-fig} shows.
Only in the limit, as $p\rightarrow0$, will $B$ reach the rightmost endpoint of the bottom edge,
denoted by $\rho$ in Figure \ref{l0bisec-fig}.
\begin{figure}
   \begin{minipage}{.45\linewidth} % [b] => Ausrichtung an \caption
      \includegraphics[width=\linewidth]{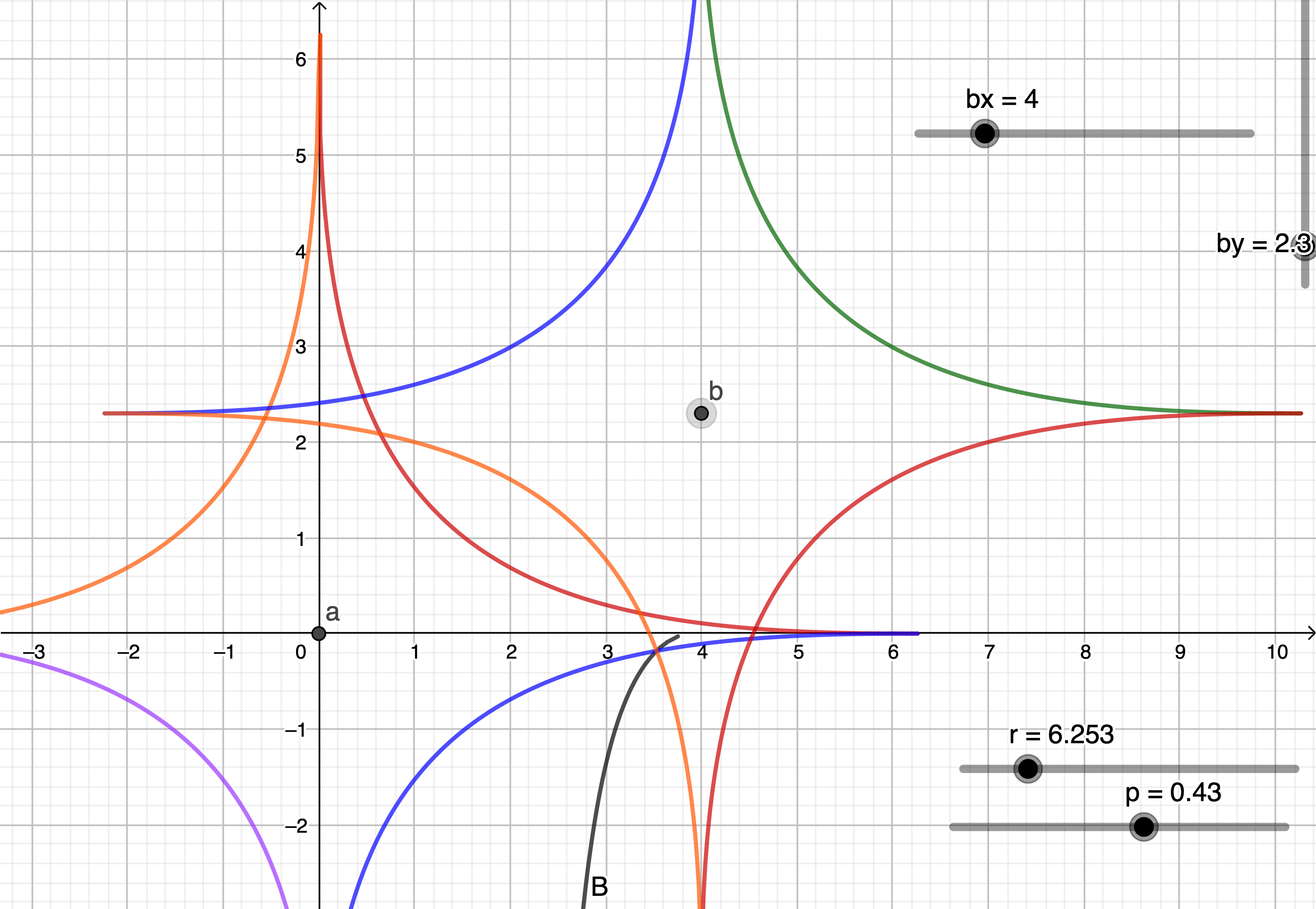}
      \caption{Two $L_{0.43}$ circles expanding further...}
      \label{expa-fig}
   \end{minipage}
   \hspace{.1\linewidth}% Abstand zwischen Bilder
   \begin{minipage}{.45\linewidth}]% [b] => Ausrichtung an \caption
      \includegraphics[width=\linewidth]{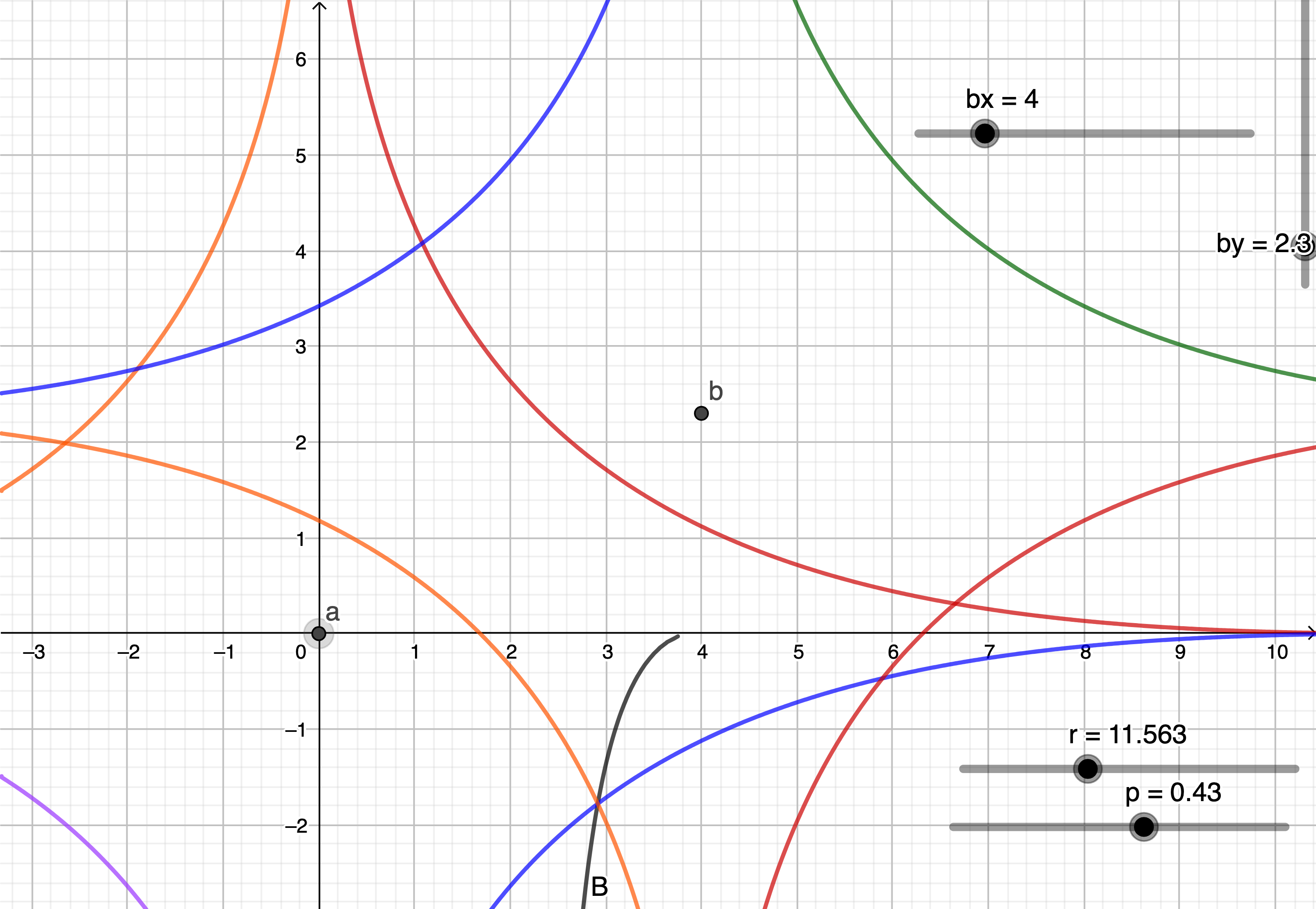}
      \caption{...while their intersections are writing out bisector segments.}
      \label{expamore-fig}
   \end{minipage}
\end{figure}

If we fix $p=0.43$ as in Figure \ref{lowerp-fig} and increase radius $r$ we can see in Figures \ref{expa-fig}   and 
\ref{expamore-fig} how bisector point $B$ moves to the left and downwards. In fact, $B$ and the other intersection points of the $L_p$ circles are writing out segments of $B_p(a,b)$ that will get closer to the curves shown in Figure \ref{Lstarbisec-fig}(a) the smaller 
$p$ becomes.

In the limit, the $L_p$ unit circle, for $p>0$, converges to the cross that consists of the four cornerpoints and the two line segments connecting the horizontal and vertical pairs.

\smallskip
For $p<0$, the $L_p$ circles of radius $r$ centred at the origin are disconnected. They consist of  four connected segments that have asymptotes a distance $r$ away from, and parallel to, the coordinate axes.
Indeed, to fulfil a  circle's equation
\[
    (|x|^p + |y|^p)^{\frac{1}{p}} =r
\]
neither $x$ nor $y$ can be zero because $0^p$ is undefined if $p<0$.
Therefore,we must have $|x| > r$; otherwise, that is, if $|x| \leq r$, we would have $|x|^p \geq r^p$ (since the function $f(z)=z^p$ is decreasing for $p<0$ and $z>0$) and we would obtain the contradiction 
$r^p = |x|^p+|y|^p > |x|^p \geq r^p$ .
If we fix $r$ and let $p$ grow to zero from below, the circle's segments strive to infinity along their asymptotes, so that 
no limits of these circles exist in the plane.

Yet the intersections of $L_p$ circles for negative $p$ define bisector curves that look quite similar to those for $p>0$; see 
Figures \ref{expaneg-fig} and \ref{expanegmore-fig}. Also, the analysis of the bisector curves for positive and negative values of $p$ in Section~\ref{ana-sect} will turn out to be almost identical.

Due to the shape of the $L_p$ circles, $L_p$\,Voronoi diagrams of $n$ points have complexity $\Omega(n^2)$ because circles
can penetrate each other in a grid-like fashion, as Seidel %\cite{s-cloVD-22}  
observed\rlap{$^{\ref{s-cloVD-22}}$}.
\begin{figure}
   \begin{minipage}{.45\linewidth} % [b] => Ausrichtung an \caption
      \includegraphics[width=\linewidth]{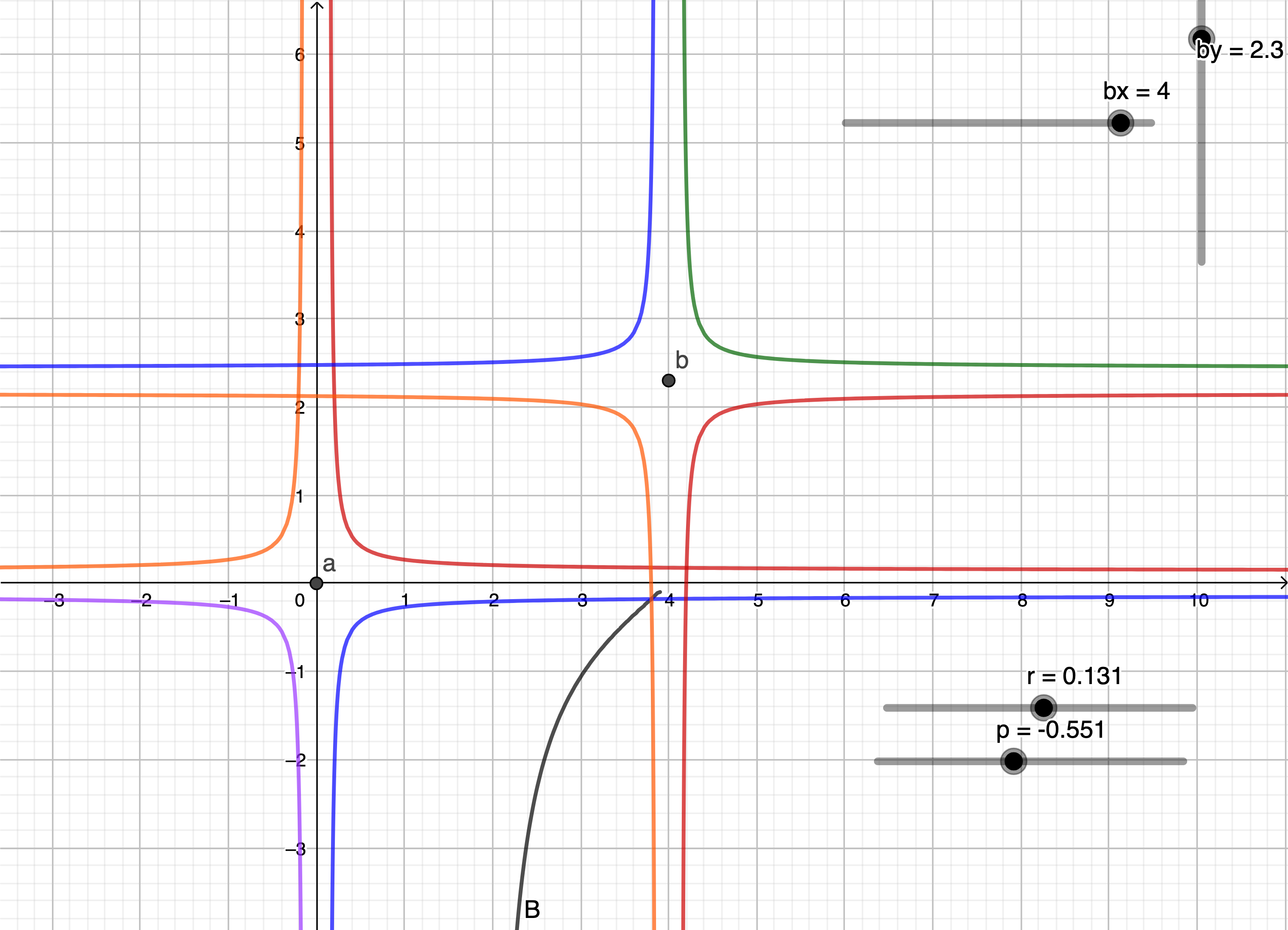}
      \caption{Two $L_{-0.55}$ circles centred at $a$ and $b$.}
      \label{expaneg-fig}
   \end{minipage}
   \hspace{.1\linewidth}% Abstand zwischen Bilder
   \begin{minipage}{.45\linewidth}]% [b] => Ausrichtung an \caption
      \includegraphics[width=\linewidth]{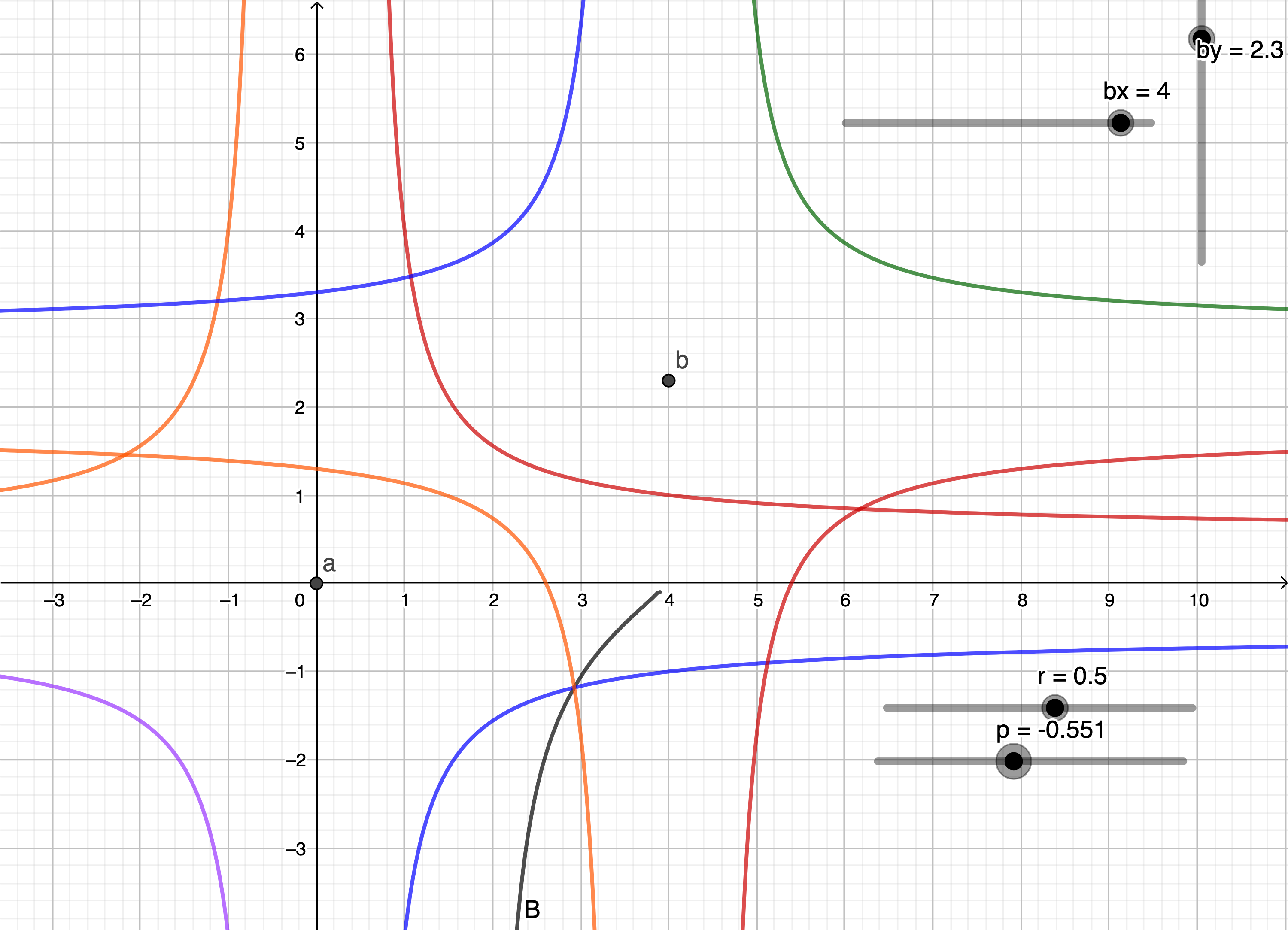}
      \caption{Bisectors look similar to the case $p>0$.}
      \label{expanegmore-fig}
   \end{minipage}
\end{figure}

It is quite interesting to study the intersection patterns of expanding $L_p$ circles 
by means of an interactive tool like GeoGebra\footnote{Our GeoGebra worksheets can be downloaded from \url{http://herman.haverkort.net/L0-Distance/}.}. 
It leads to the following observation, whose proof will be given in Section~\ref{existence-sec}, after introducing necessary notations.
\begin{figure}
\begin{center}
\includegraphics[scale=0.6]{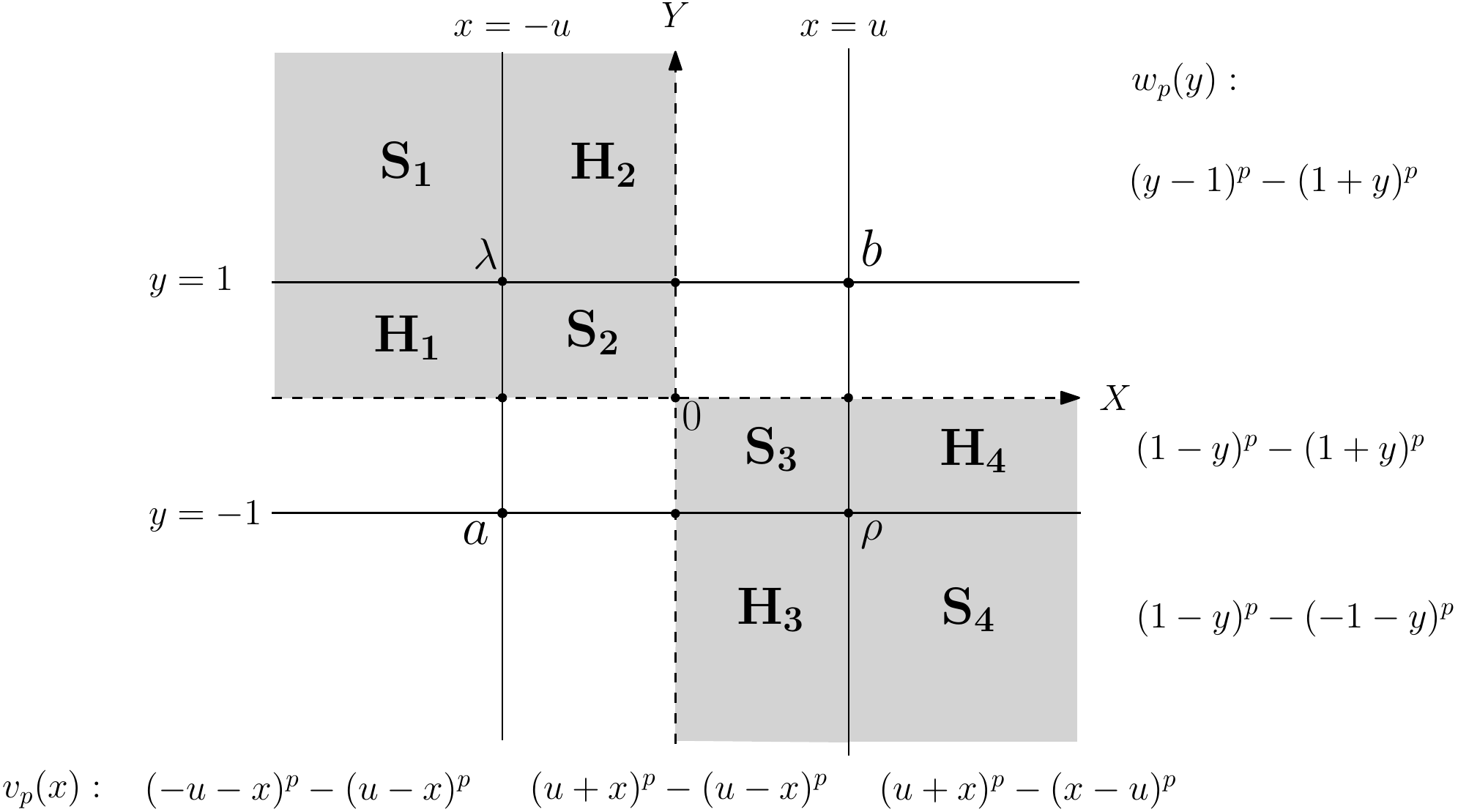}
\caption{The cells in the plane where segments of $B_p(a,b)$ appear.
Notations $H$ and $S$ denote cells that contain hyperbola and straight line segments, respectively. Also shown are the functions $v_p(x)$ and $w_p(y)$ for the various columns and rows of the grid.}
\label{cell-fig}
\end{center}
\end{figure}
\begin{lemma}\label{cell-lem}
If $x$ lies in the open $X$-interval of a cell $C$ coloured grey in Figure \ref{cell-fig} then, for each $p$ close enough
to zero, there exists $y_p$ in the open $Y$-interval of $C$ such that $(x,y_p)$ belongs to $B_p( a,b)$.
White cells do not contain points of any $B_p( a,b)$.
\end{lemma}
The reason for restricting ourselves to open intervals are the asymptotes involved,
and also the fact that  points  like the lower right corner of the bounding box of $a,b$
in Figures \ref{lowp-fig} and \ref{lowerp-fig}  are not situated on any $B_p(a,b)$ with $p > 0$; only in the $L_0$ diagram and in the $L_p$ diagrams for $p < 0$ does this point appear as a Voronoi vertex.
Bisector points on the vertical lines at $x \in \{-u, 0, u\}$ 
and on the horizontal lines at $y \in \{-1, 0, 1\}$ will be addressed after the proof of Lemma \ref{cell-lem}.

%%%%%%%%%%%%%%%%%%%%%%%%%%%%%%%%%%%%%%%%%%%%%%%%%%%%%
\section{Analysis}\label{ana-sect}
%%%%%%%%%%%%%%%%%%%%%%%%%%%%%%%%%%%%%%%%%%%%%%%%%%%%%

\subsection{Framework}
Throughout this section we will assume that $|p|<1$ is sufficiently small, but not zero.

We consider two point sites $a$ and $b$ that do not lie on a common horizontal or vertical line. If we subject these points to translation, reflection in a coordinate axis, or reflection in the line $y = x$, the bisector under the $L_p$ distance is subject to the same transformation. Therefore, to investigate the shape of the bisector, it suffices to consider two point sites 
$a=(-u, -1)$ and $b=(u, 1)$ where $u \geq 1$. Note that under these conditions, Equations \ref{linedef} and \ref{hyperboladef} of the linear and the hyperbolic sections of the $L_*$ bisector $B_*(a,b)$ simplify to $y = -x/u$ and $y = -u/x$, respectively.

In this subsection, we set up the general framework to analyse the shape of the bisector in all grey cells of Figure~\ref{cell-fig}. In the next subsections, we analyse the $H$-cells and the $S$-cells, respectively.

The $L_p$ bisector $B_p(a,b)$ of $a$ and $b$ is the locus of all points $(x,y)$ in the plane satisfying
\begin{eqnarray}
    (|x+u|^p + |y+1|^p)^{1/p} &=& (|x-u|^p + |y-1|^p)^{1/p},\nonumber\\
    \llap{\hbox{that is,\quad\quad\ }}|x+u|^p + |y+1|^p\hphantom{)^{1/p}} &=& \hphantom{(}|x-u|^p + |y-1|^p. \label{mitabs}
\end{eqnarray}

\begin{lemma}\label{conv-theo}
Let $(x,y_p)_p$, where $p$ tends to zero, be a sequence of points of $B_p(a,b)$ in one of the cells coloured grey in Figure \ref{cell-fig}.
Let $h$ be the hyperbola $h(x) = -u/x$, and let $s$ be the line $s(x) = -x/u$.
Then $\lim_{p\rightarrow0} y_p = h(x)$ holds in cells of type $H$, and 
$\lim_{p\rightarrow0} y_p = s(x)$ in cells of type $S$.
\end{lemma}
Let us define 
\begin{eqnarray*}
     w_p(y) &:=& |y-1|^p - |y+1|^p  \mbox{\ \ and} \\
     v_p(x) &:=& |x+u|^p - |x-u|^p.
\end{eqnarray*}
Then the bisector $B_p(a,b)$ is given by $w_p(y) = v_p(x)$, because of Equation \ref{mitabs}.

How are these equations related to the hyperbola $h$ and the line $s$?
A connection is provided by the following functional properties.
\begin{lemma} \label{factorout-lem}
For all $p$, we have
\begin{eqnarray*}
   w_p(h(x)) &=& \frac{1}{|x|^p} \ v_p(x)  \\
   w_p(s(x))  &=& \frac{1}{|u|^p} \ v_p(x).
\end{eqnarray*}
\end{lemma}
\begin{proof}
We have
\begin{eqnarray*}
   w_p(h(x)) &=& \left|-\frac{u}{x} - 1\right|^p   \  -  \   \left|-\frac{u}{x} + 1\right|^p  \  
                  = \ \left|1+\frac{u}{x}\right|^p  \   - \   \left|1-\frac{u}{x}\right|^p    \mbox{\ \ \ and}  \\
   w_p(s(x))  &=& \left|-\frac{x}{u} - 1\right|^p   \  -  \  \left|-\frac{x}{u} + 1\right|^p   \ 
                 = \  \left|\frac{x}{u}+1\right|^p   \  -  \  \left|\frac{x}{u}-1\right|^p    \mbox{\ \ \ as well as}  \\
   v_p(x)     &=& |x|^p \ \biggl(  \left|1+\frac{u}{x}\right|^p  \   - \   \left|1-\frac{u}{x}\right|^p   \biggr)   \  
                 = \   |u^ p| \ \biggl(  \left|\frac{x}{u}+1\right|^p   \  - \    \left|\frac{x}{u}-1\right|^p   \biggr).
 \end{eqnarray*}
\end{proof}
Lemma \ref{factorout-lem} has a useful consequence. 
\begin{lemma}\label{err-lem}
Let $b=(x,y_p)$ be a bisector point of $B_p(a,b)$. Then,
\begin{eqnarray*}
w_p(h(x)) - w_p(y_p) &=& \left(\frac{1}{|x|^p} -1\right)\ v_p(x)  \\
w_p(s(x)) - w_p(y_p) &=& \left(\frac{1}{|u|^p} -1\right)\ v_p(x) .
\end{eqnarray*}
\end{lemma}

The idea of our proof for Lemma~\ref{conv-theo} is as follows. We use Lemma \ref{err-lem} to provide an upper bound to the error we incur
in bisector equation~\ref{mitabs} if we replace the true value $y_p$ by the approximate value $h(x)$ or $s(x)$. Then, we 
apply the mean value theorem to the function $w_p(y)$, which is continuous and everywhere differentiable except at $y=\pm 1$,
to upper-bound $|y_p - h(x)|$, or $|y_p -s(x)|$, by said error. This leads to the following estimates.

\begin{lemma}\label{mainform-lem} Define $ z_p(y):=  \frac{1}{p} \frac{d}{dy}w_p(y)$. If $\{y_p, h(x)\}$ or $\{y_p, s(x)\}$ are both situated in the same $Y$-interval $(-\infty, -1), (-1,1)$, or $(1, \infty)$, and $x \not=0$, then:
\begin{eqnarray*}
|y_p-h(x)| &\leq &    |z_p(y^{*})|^{-1} \cdot o(1)  \mbox{\rm \ \  for some } y^{*} \mbox{\rm \   between } y_p \mbox {\rm \ and } h(x)\\
|y_p-s(x)| &\leq &    |z_p(y^{*})|^{-1} \cdot o(1)  \mbox{\rm \ \  for some } y^{*} \mbox{\rm  \  between } y_p \mbox {\rm \ and } s(x) \label{mainform},
\end{eqnarray*}
where $o(1)$ denotes a function $f$ of $x$ and $p$ such that for all $x \neq 0$, we have $\lim_{p\rightarrow 0} f(x,p) = 0$. 
\end{lemma}
\begin{proof}
By the mean value theorem, there exists $y^{*}$ between $y_p$ and $h(x)$ such that
\begin{eqnarray}
 y_p - h(x)  &=&   \frac{1}{\frac{d}{dy}w_p(y^{*})} \cdot (w_p(y_p) - w_p(h(x))) \nonumber\\
 & = &  \frac{1}{z_p(y^{*})} \cdot \frac{1}{p}\,  (w_p(y_p) - w_p(h(x))) \nonumber\\
                  &=& \frac{1}{z_p(y^{*})} \cdot \frac{1}{p} \, \left(1-\frac{1}{|x|^p}\right) \cdot v_p(x) \label{factors},    
\end{eqnarray}
by Lemma \ref{err-lem}. In
\[
     \frac{\left(1-\frac{1}{|x|^p}\right)}{p}
\]
both numerator and denominator converge to zero as $p$ tends to zero. By L'Hospital's rule, 
the limit of this fraction equals the limit
of the fraction of the derivatives with respect to $p$.  While the denominator has derivative 1, the numerator's 
 derivative is
 \[
      \frac{d}{dp}\left(1-\frac{1}{|x|^p}\right) \ = \ -\frac{-\frac{d}{dp} |x|^p}{|x|^{2p}} \ =\  \frac{|x|^p\cdot \ln |x|}{|x|^{2p}} \  = \ |x|^{-p}\cdot \ln |x|,
 \]
which goes to $\ln |x|$. Since $v_p(x)$ in equation \ref{factors} tends to zero with $p$, so does its product with 
the convergent factor $\frac{1}{p} \, (1-\frac{1}{|x|^p})$.
This proves the claim for $h(x)$, with $f(x,p) := \frac{1}{p} \, (1-\frac{1}{|x|^p})\, v_p(x)$. The case of $s(x)$ is analogous.
\end{proof}
While both estimates in Lemma \ref{mainform-lem} may apply to a given value $x$, the two values $y_p$ will be situated in 
different $Y$-intervals because $s(x)$ and $h(x)$ are; see Figure \ref{cell-fig}.

\smallskip
The remaining challenge is in establishing an upper bound on $|z_p(y^{*})|^{-1} $ that 
converges as $p$ tends to 0 (the limit of convergence may depend on $x$). 
This task will require some case analysis, because 
different signs of the expressions 
within $|\cdot|$  in Equation \ref{mitabs}, and the coordinate axes, split the plane into several cells shown in Figure \ref{cell-fig}. 
Of the cells coloured grey, those of type $S$ contain a segment of line $s(x)=-x/u$,  and those labelled $H$, a segment of hyperbola $h(x)=-u/x$.
Figure \ref{cell-fig} also shows the equations for $v_p(x)$ and $w_p(y)$ in their respective columns and rows of cells. 

%%%%%%%%%%%%%%%%%%%%%%%%%%%%%%%%%%%%%%%%%%%%%%%%%%%%%%
\subsection{Hyperbolae}\label{hyp-ssect}
%%%%%%%%%%%%%%%%%%%%%%%%%%%%%%%%%%%%%%%%%%%%%%%%%%%%%%
We begin our case analysis with the cells \region{H_i} for which we want to prove that $B_p(a,b)$  converges to $y=h(x)$, using
Lemma \ref{mainform-lem}.

First, let $(x,y_p)_p$, for $p\rightarrow0$, be a sequence of points on $B_p(a,b)$ in \region{H_2}.
\begin{figure}
\begin{center}
\includegraphics[scale=0.5]{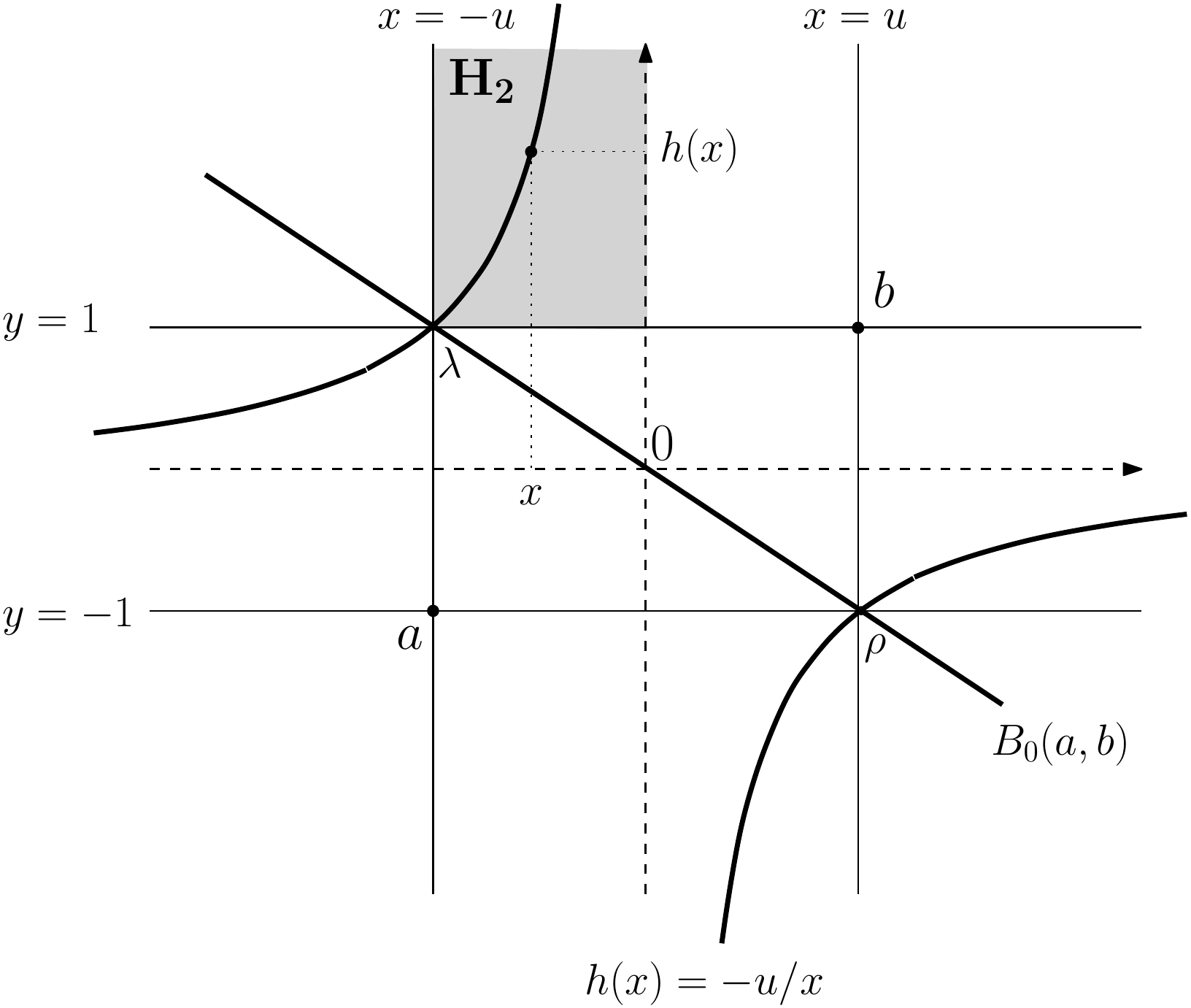}
\caption{A bisector of two points $a,b$ in the $L_0$ diagram. Their Voronoi regions look as indicated in Figure \ref{vorosketch-fig}.
The number $h(x)$ denotes the $Y$-coordinate of the point marked.}
\label{l0bisec-fig}
\end{center}
\end{figure}
Note that inside \region{H_2} we always have $-u < x < 0$, and $h(x), y_p > 1$; see Figure~\ref{l0bisec-fig}.
Moreover, Equation \ref{mitabs} becomes
\begin{eqnarray*}
  v_p(x)  \, = \,  (u+x)^p - (u-x)^p \ = \ (y-1)^p - (y+1)^p \, = \, w_p(y),
 \end{eqnarray*}
 compare Figure \ref{cell-fig}. We want to find an upper bound on $|z_p(y^{*})|^{-1}$, where
\begin{eqnarray}
 |z_p(y)|^{-1}  \, = \,  \frac{1}{ (y -1)^{p-1} -  (y + 1)^{p-1} } \label{whereto}
 \end{eqnarray}
 and $y^{*}$ lies between $y_p$ and $h(x)$. To this end, we first observe:
 
 \begin{lemma}\label{lowbo}
For all $p$ close enough to zero, $y_p$ and $h(x)$ both lie in the open interval $(1,  2\, h(x) + 3)$.
\end{lemma}
\begin{proof}
$y_p$ and $h(x)$ are both greater than 1, as we are considering \region{H_2}. Obviously, $h(x) < 2 h(x) + 3$. It remains to show $y_p < 2 h(x) + 3$. 
Let us rewrite the bisector equation for $(x,y_p)$,
\begin{eqnarray*}
    (u+x)^p -(u-x)^p &=& (y_p-1)^p - (y_p+1)^p, \mbox{\ \ \ as} \\
   j(x)-j(-x) &=& m(1)-m(-1)
 \end{eqnarray*}
with functions $j(z):=(u+z)^p$ and $m(z):= (y_p-z)^p$. Applying the mean value theorem to both sides yields%
\footnote{We can write $j'(z)$ since it is clear that the derivative is taken with respect to $z$.}
\begin{eqnarray*}
   2x\, p\,  (u+z_1)^{p-1} \ =\  2x\, j'(z_1) &=&  2\, m'(z_2) \  =\  -2\, p\,  (y_p - z_2)^{p-1}    
 \end{eqnarray*}
 with $z_1 \in (-x,x)$ and $z_2 \in (-1,1)$. If $|p|$ is small enough to guarantee $(-x)^{\frac{1}{p-1}} < \frac{2}{-x}$, then cancelling out $-2p$, raising both sides to power $\frac{1}{p-1}$, and solving for 
 $y_p$ leads to 
 \begin{eqnarray*}
    y_p &=&  (-x)^{\frac{1}{p-1}} \, (u+z_1) - z_2   \\
           &<&  (-x)^{\frac{1}{p-1}} \, (u-x) + 1 \\
           &<&  2\,\frac{u}{-x} + 3 \\
           &=& 2\,h(x) + 3.
\end{eqnarray*}
\end{proof}

Furthermore, $|z_p(y)|^{-1}$ increases with $y$, as the denominator in Equation~\ref{whereto} decreases in $y$. Indeed, 
\[
     \frac{d}{dy} \biggl( (y -1)^{p-1} -  (y + 1)^{p-1} \biggr) \, =\,  (p-1) \biggl(   (y - 1)^{p-2}  -   (y + 1)^{p-2}  \biggr) \ < 0
 \]
holds since $p-1$ is negative and $(y - 1)^{p-2}  -   (y + 1)^{p-2}$ is positive for $y > 1$ and $p < 1$. 

With Lemma~\ref{lowbo} it follows that $|z_p(y^*)|^{-1} < |z_p(y)|^{-1}$ for $y = 2 h(x) + 3$. Thus we obtain:
\begin{eqnarray*}
|z_p(y^*)|^{-1} & < & |z_p(2 h(x) + 3)|^{-1} \\
& = & \frac{1}{ (2h(x) + 2)^{p-1} -  (2 h(x) + 4)^{p-1} } \\
& = & \frac{ (2 h(x) + 2)^{1-p} \cdot (2 h(x) +4)^{1-p} }{ (2 h(x) + 4)^{1-p} -  (2 h(x) + 2)^{1-p} }.
\end{eqnarray*}

This bound converges (namely, to $(2 h(x) + 2)(2 h(x) + 4)/2$) as $p$ tends to zero. Thus, with Lemma~\ref{mainform-lem} it follows that sequence $(x,y_p)_p$ in \region{H_2}  converges to the hyperbola point $(x,h(x))$ for $p\rightarrow0$.

\bigskip
Next, we consider the hyperbola in cell \region{H_4}. In here, the bisector equation reads
\[
   v_p(x) \, = \, (x+u)^p - (x-u)^p \ =\  (1-y)^p - (1+y)^p\, = \, w_p(y),
\]
and we want to establish an upper bound on $|z_p(y^{*})|^{-1}$, where
\begin{eqnarray}
 |z_p(y)|^{-1}  \, = \,  \frac{1}{ (1 + y)^{p-1} +  (1 - y)^{p-1} } \label{wheretoH4}
 \end{eqnarray}
and $y^{*}$ lies between $y_p$ and $h(x)$, which, by definition of \region{H_4}, both lie in the open interval $(-1,1)$. 

If $y$ in Equation~\ref{wheretoH4} gets close to $1$, the term $(1-y)^{p-1} $ tends to infinity, so that $|z_p(y)|^{-1}$ goes to zero.
The same holds for $y$ close to $-1$. Between these extremes, that is, inside the open interval $(-1,1)$, the derivative
\[
      \frac{d}{dy} \, \frac{1}{ (1+y)^{p-1} +  (1-y)^{p-1} }  \ = \ 
      (p-1) \, \frac{ (1-y)^{p-2}  - (1+y)^{p-2} }{ ( (1+y)^{p-1} +  (1-y )^{p-1} )^2  }
\]
has a unique zero for $1-y = 1+y$, that is, at $y=0$, where   $|z_p(y)|^{-1}$ takes on its maximum value $\frac{1}{2}$. Thus we get $|z_p(y^{*})|^{-1} \leq |z_p(0)|^{-1} = \frac12$, which, with Lemma~\ref{mainform-lem}, completes the proof that the sequence $(x,y_p)_p$ in \region{H_4}  converges to the hyperbola point $(x,h(x))$ for $p\rightarrow0$.

\medskip
The analysis of cells \region{H_1} and \region{H_3} is symmetric.%

%%%%%%%%%%%%%%%%%%%%%%%%%%%%%%%%%%%%%%%%%%%%%%%%%%%%%%
\subsection{Lines}\label{line-ssect}
%%%%%%%%%%%%%%%%%%%%%%%%%%%%%%%%%%%%%%%%%%%%%%%%%%%%%%
Now we show that the part of $B_p(a,b)$ within cells \region{S_i} converges to the line $s(x)=-x/u$ passing through vertices $\lambda$ and $\rho$ of the bounding box of $a,b$. First, we consider cell  \region{S_1}; see Figure \ref{l0bisec3-fig}.
\begin{figure}
\begin{center}
\includegraphics[scale=0.5]{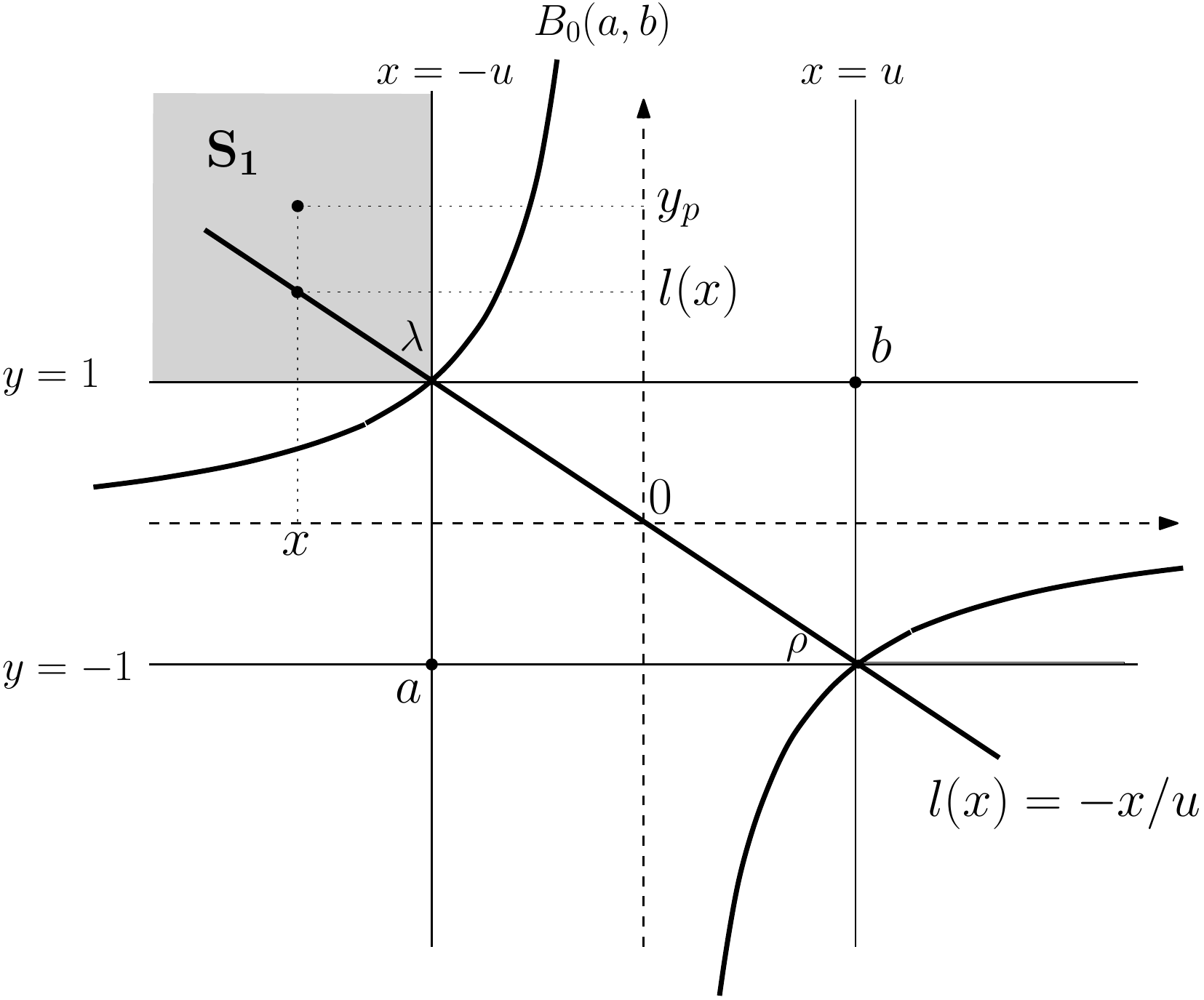}
\caption{The line $s(x)=-x/u$ to the left of vertex $\lambda$.}
\label{l0bisec3-fig}
\end{center}
\end{figure}
Here the bisector equation  reads
\begin{eqnarray*}
 v_p(x) \, =\,  (-u-x)^p \, - \,  ( u-x)^p \ =\  (y-1)^p -(y+1)^p\,  = \, w_p(y), %\label{biseceq}
\end{eqnarray*}
which results in 
\[
|z_p(y)|^{-1}  \, = \,  \frac{1}{ (y-1)^{p-1} -  (y +1)^{p-1} }  
\]
This expression is the same as Equation \ref{whereto} in the discussion of cell \region{H_2}. We can deal with it in the same way, replacing $h(x)$ by $s(x)$ and with $x$ and $u$ changing roles accordingly. For completeness, we include the proof of the resulting variation on Lemma~\ref{lowbo}:
 
\begin{lemma}\label{lowbo2}
For all $p$ close enough to zero we have $y_p < 2 \, s(x) +3 $.
\end{lemma}
\begin{proof}
Here, we rewrite the equation for $(x,y_p)$ 
\begin{eqnarray*}
    (-u-x)^p -(u-x)^p &=& (y_p-1)^p - (y_p+1)^p \mbox{\ \ \ as} \\
   j(u)-j(-u) &=& m(1)-m(-1)
 \end{eqnarray*}
with functions $j(z):=(-z-x)^p$ and $m(z):= (y_p-z)^p$. Applying the mean value theorem to both sides yields
\begin{eqnarray*}
   2u \cdot (-p)\,  (-z_1-x)^{p-1} \ =\  2u\, j'(z_1) &=&  2\, m'(z_2) \  =\  2\, (-p)\,  (y_p -z_2)^{p-1}    
 \end{eqnarray*}
with $z_1 \in (-u,u)$ and $z_2 \in (-1,1)$. 
If $|p|$ is small enough to guarantee $u^{\frac{1}{p-1}} < \frac{2}{u}$,
then cancelling out $-2p$, raising both sides to power $\frac{1}{p-1}$, and solving for $y_p$ leads to 
 \begin{eqnarray*}
     y_p &=&  u^{ \frac{1}{p-1}  } \, (-z_1-x) + z_2   \\
           &<&  u^{\frac{1}{p-1}} \, (u-x) + 1 \\
           &<&  2\frac{-x}{u} + 3 \\
           &=& 2\, s(x) +3
 \end{eqnarray*}
\end{proof}

The rest of the proof follows the same reasoning as for \region{H_2}, with the upper bound on $|z_p(y^*)|^{-1}$ converging to $(2(s(x)+2)(2(s(x)+4)/2$. With Lemma~\ref{mainform-lem} it follows that sequence $(x,y_p)_p$ in \region{H_1} converges to the point $(x,s(x))$ for $p\rightarrow0$. 
For cell \region{S_4}, the discussion follows the same line. 

So let us, finally, look at cell \region{S_2} in the centre,
where the bisector equation reads
\begin{eqnarray*}
 v_p(x) \, =\,  (u+x)^p \, - \,  ( u-x)^p \ =\  (1-y)^p -(1+y)^p\,  = \, w_p(y). 
\end{eqnarray*}
We obtain
\begin{eqnarray*}
|z_p(y)|^{-1}  \, = \,  \frac{1}{ (1 + y)^{p-1} +  (1 - y)^{p-1} }\ , %\label{wheretoS2}
\end{eqnarray*}
the same as in \region{H_4}. As calculated in Section~\ref{hyp-ssect}, it follows that $|z_p(y^*)|^{-1}  < \frac{1}{2}$ holds and therefore, by Lemma~\ref{mainform-lem}, the sequence $(x,y_p)_p$ in \region{S_2} converges to the point $(x,s(x))$ for $p \rightarrow 0$. The same applies to \region{S_3}.

\smallskip
This completes the proof of Lemma~\ref{conv-theo}.

\subsection{Proof of existence}\label{existence-sec}
It remains to prove Lemma \ref{cell-lem} of Section  \ref{circ-sect} that ensures  the existence of the bisector points
$(x,y_p)_p$ referred to in Lemma~\ref{conv-theo}. Again, it is helpful to employ a tool like GeoGebra to visualise the graphs
of the functions $v_p$ and $w_p$ on their respective intervals.

\medskip
\begin{proof} (of Lemma \ref{cell-lem})
We refer to Figure \ref{cell-fig} and the notations given there. We shall study the bisector $B_p(a,b)$ to the right of the vertical axis; then our arguments carry over to the negative $X$-range by symmetry.

First observe that \emph{all} points $(x,y)$
within the upper right quadrant are closer to $b$ than to $a$ with respect to \emph{both} coordinates. Thus, the white cells in the upper right quadrant belong to the Voronoi region of $b$ and cannot contain any bisector points of $B_p(a,b)$. 

It remains to prove that for any $x \in (0,u)$, the cells \region{H_3} and \region{S_3} each contain a bisector point of $B_p(a,b)$ on the vertical line at $x$, and for any $x \in (u,\infty)$, the cells \region{S_4} and \region{H_4} each contain a bisector point of $B_p(a,b)$ on the vertical line at $x$. This time, the cases $p > 0$ and $p < 0$ require different arguments.

\bigskip
We first consider the case $p > 0$. 
If $x \in (0,u)$, then $v_p(x) = (u+x)^p - (u-x)^p$;
if $x \in (u,\infty)$, then $v_p(x) = (x+u)^p - (x-u)^p$. Thus, for any $x \in (0,u) \cup (u,\infty)$, it holds that $v_p(x) < 1$ once $p$ is small enough.

Now let $x \in (0,u) \cup (u,\infty)$ be fixed.
Then, for the lowest $Y$-interval $(-\infty,-1)$ we have $w_p(y)= (-y+1)^p -(-y-1)^p$, hence%
\footnote{One can apply L'Hospital's rule to  $w_p(y) \, =\,  \frac{ (\frac{-y+1}{-y-1})^p -1  }{ (-y-1)^{-p}   }$ to show
that $w_p(y)$ goes to zero for $-\infty\leftarrow y$.}
\[
     w_p((-\infty,-1)) \ = \ (0,2^p).
\] 
Once $p$ is small enough to ensure $v_p(x) < 1$, there will always be $y_p$ in $(-\infty,-1)$ such that
$v_p(x)=w_p(y_p)$ holds, because $2^p$ goes to $1$ from above. Thus, there is a point $(x,y_p)$ of $B_p(a,b)$ in \region{H_3} for any $x \in (0,u)$ and there is a point $(x,y_p)$ of $B_p(a,b)$ in \region{S_4} for any $x \in (u,\infty)$.

For $y$ in the $Y$-range $(-1,0)$, we have $w_p(y)= (1-y)^p - (1+y)^p$ and $w_p((-1,0 )) = (0,2^p)$, so that the same arguments apply for regions \region{S_3} and \region{H_4}.

\bigskip
Now consider the case of $p < 0$.
Function $v_p(x)$ is now negative and goes to $-\infty$ at $x=u$, whereas $v_p(x)$ goes to 0 for $x=0$ and $x\rightarrow\infty$. Thus we have 
$v_p((0,u)) = v_p((u,\infty)) = (-\infty, 0)$.
On the other hand, we have $w_p((-\infty,-1 )) = (-\infty,0)$, with a pole at $y=-1$,
covering every possible value of $v_p(x)$. Thus, for any $x$ in $(0,u)$ or $(u,\infty)$ there is a $y_p \in (-\infty,-1)$ such that $v_p(x) = w_p(y_p)$ holds, proving the existence of a bisector point at $x$ in \region{H_3} or \region{S_4}, respectively.

The case of $Y$-interval $(-1,0)$ for the cells \region{S_3} and \region{H_4} is similar: we have $w_p((-1,0))=(-\infty,0)$.

This completes the proof of Lemma \ref{cell-lem}.
\end{proof}

Now, let us discuss what bisector 
points are situated on the vertical lines at
$x \in \{-u,0,u\}$ 
and the horizontal lines at $y \in \{-1,0,1\}$ that separate the cells that were considered so far.

If $x = 0$, the bisector equation~\ref{mitabs} becomes
\[
   |y+1|^p = |y-1|^p,
\]
which is satisfied by the single value $y=0$, which proves that for any $p \neq 0$, the bisector $B_p(a,b)$ contains the origin (which also lies on $B_*(a,b)$), but no other point on the vertical axis. By an analogous calculation, the bisector contains no other point on the horizontal axis.

Now consider the horizontal line $y = -1$. For points on this line, 
and for $p>0$,
the bisector equation \ref{mitabs} becomes
\[
\left|\frac{x+u}{2}\right|^p - \left|\frac{x-u}{2}\right|^p = 1.
\]
For fixed $p \in (0,1)$, 
the left-hand side, as a function of $x$, 
is negative and decreasing for $x < -u$,
drops to $-u^p$ at $x = -u$,
rises to $u^p$ at $x = u$, 
and then decreases towards its asymptotic value $0$.
Thus, the line $y = -1$ always contains two points $(x_p,-1)$ and $(x'_p,-1)$ of $B_p(a,b)$, where $-u < x_p < u < x'_p$.
As $p$ tends to $0$, we find that $x_p$ and $x'_p$ both converge towards $u$. Indeed, as $p$ tends to 0, the maximum $u^p$ at $x = u$ of the left-hand side converges towards the value of the right-hand side, namely 1, whereas the value of the left-hand side converges towards $-1$ for $x = -u$ and towards 0 for any other $x$. Thus, as $p$ tends to 0 from above, the intersection of $B_p(a,b)$ with the line $y = -1$ converges to the point $(u,-1)$.

If $p$ tends to zero from below, then 
the line $y=-1$ does not contain any points of $B_p(a,b)$ except $(u,-1)$, because by substituting $y=-1$ in the bisector equation 
\[
L((x+u, y+1)) = (|x+u|^p + |y+1|^p)^{1/p} \ =\  (|x-u|^p + |y-1|^p)^{1/p} = L((x-u, y-1))
\]
the left-hand side becomes zero, 
by our continuation of $L_p$ to arguments of value zero, and the right-hand side $(|x-u|^p + 2^p)^{1/p}$
can only be zero for $x=u$.
Thus, the intersection of $B_p(a,b)$ with the line $y = -1$ converges to the point $(u,-1)$ as $p$ tends to zero from above or below. 

By symmetric arguments, the intersection of the $B_p(a,b)$ with the line $y = 1$ converges to the point $(-u,1)$ as $p$ tends to zero.

Now consider the vertical line $x = u$. 
For $p>0$, the bisector equation \ref{mitabs} becomes
\[
\left|\frac{y+1}{2u}\right|^p - \left|\frac{y-1}{2u}\right|^p = -1.
\]
The left-hand side, as a function of $y$, has its unique minimum $-1/u^p$ at $y = -1$, so (recall $u \geq 1$) the equation cannot be fulfilled and no points of $B_p(a,b)$ are on the line $x = u$ for any $p > 0$, unless $u = 1$ holds.%
\footnote{Note this means that in Figure~\ref{vorosketch-fig}, right, the two red faces on the left are actually a single face; only in the limit as $p$ drops to zero, will the boundaries of the two upper green faces really touch each other and separate the two red faces on the left from each other.}
If $u=1$ or $p<0$, then $(u,-1)$ is the only point 
of $B_p(a,b)$ on the line $x = u$.

By symmetric arguments, the intersection of $B_p(a,b)$ with the line $x = -u$ is empty for any $p > 0$ and $u > 1$, and contains only the point $(-u,1)$ for $p < 0$ or $u = 1$.

Thus, as $p$ approaches 0, the points of $B_p(a,b)$ that are
on the lines $x \in \{-u, 0, u\}$ and $y \in \{-1, 0, 1\}$ converge to $(-u,1)$, the origin, and $(u,-1)$, which all lie on the line $y = -x/u$ of $B_*(a,b)$. Together with Lemma~\ref{conv-theo}, this implies that $B_p(a,b)$ converges to exactly the complete line $y = -x/u$ and the complete hyperbola $y = -u/x$ that make up $B_*(a,b)$.

\bigskip
It remains to discuss the Voronoi regions $\lim_{p\rightarrow 0} V_p(a,b)$ and $\lim_{p\rightarrow 0} V_p(b,a)$. As the $L_p$ measures are continuous for any $p \neq 0$, the Voronoi regions must consist of entire faces of the subdivision of the plane that is induced by $B_p(a,b)$. To complete the proof of Theorem~\ref{main-theo}, we only need to show that, regardless of whether $p$ tends to zero from above or from below, $V_p(a,b)$ converges to the same selection of three faces of the subdivision induced by $B_*(a,b)$, namely the faces that constitute $V_*(a,b)$. Because the $L_p$ distances are continuous, it suffices to verify this on the basis of one point in the interior of each face.

In particular, consider the points $a = (-u,-1)$, $\gamma = (-u,2)$ and $\delta = (2u,-1)$. These points obviously belong to $V_*(a,b)$ and to $V_p(a,b)$ if $p < 0$, since their distance to $a$ is zero whereas their distances to $b$ are positive. When $p$ approaches zero from above, these points also belong to $V_p(a,b)$: the point $a$ trivially so; the point $\gamma$ because the bisector does not intersect the vertical line that contains $a$ and $\gamma$ (as we have seen above); and, finally, the point $\delta$ because:\[\begin{array}{rcll}
L_p(\delta - a) & < & L_p(\delta - b) & \Leftrightarrow \\
(3u)^p & < & u^p + 2^p & \Leftrightarrow \\
3^p - 1 & < & (2/u)^p
, & \\
\end{array}\]
which holds for all positive $p$ close enough to zero, as the left-hand side converges to 0 while the right-hand side converges to 1.

\smallskip
This completes the proof of Theorem \ref{main-theo}.

\end{document}